\documentclass[letterpaper]{amsart}

\usepackage[initials,nobysame]{amsrefs}
\usepackage{amssymb,mathtools,mathrsfs,tikz}
\usepackage{hyperref} 
\usepackage[shortlabels]{enumitem}

\BibSpec{collection.article}{%
	+{}  {\PrintAuthors}				{author}
	+{,} { \textit}                  	{title}
	+{.} { }                         	{part}
	+{:} { \textit}                  	{subtitle}
	+{,} { \PrintContributions}      	{contribution}
	+{,} { \PrintConference}         	{conference}
	+{}  {\PrintBook}                	{book}
	+{,} { }                         	{booktitle}
	+{,} { }						 	{series}
	+{,} { }						 	{publisher}
	+{,} { }						 	{address}
	+{,} { \PrintDateB}              	{date}
	+{,} { pp.~}                     	{pages}
	+{,} { }                         	{status}
	+{,} { \PrintDOI}                	{doi}
	+{,} { available at \eprint}     	{eprint}
	+{}  { \parenthesize}            	{language}
	+{}  { \PrintTranslation}        	{translation}
	+{;} { \PrintReprint}            	{reprint}
	+{.} { }                         	{note}
	+{.} {}                          	{transition}
	+{}  {\SentenceSpace \PrintReviews} {review}
}

\theoremstyle{plain}
\newtheorem{theorem}{Theorem}
\newtheorem{lemma}[theorem]{Lemma}

\newtheorem{prop}[theorem]{Proposition}
\newtheorem{corollary}[theorem]{Corollary}

\newtheorem{claim}{Claim}

\theoremstyle{definition}

\theoremstyle{remark}
\newtheorem{remark}[theorem]{Remark}

\numberwithin{equation}{section}
\numberwithin{theorem}{section}
\numberwithin{conjecture}{section}

\newcommand{\br}{\overline}
\newcommand{\C}{\mathbb C}
\newcommand{\D}{\mathbb D}
\newcommand{\N}{\mathbb N}
\newcommand{\R}{\mathbb R}
\DeclareMathOperator{\dist}{dist}
\DeclareMathOperator{\diam}{diam}
\DeclareMathOperator{\inter}{int}
\DeclareMathOperator{\cig}{Cig}
\renewcommand{\mod}{\mathrm{Mod\,}}

\title{Uniformization of Gromov hyperbolic domains by circle domains}

\author{Christina Karafyllia}
\author{Dimitrios Ntalampekos}
\address{Mathematics Department, Stony Brook University, Stony Brook, NY 11794, USA.}

\thanks{The second author is partially supported by NSF Grant DMS-2246485 and the Simons Foundation.}
\email{xristinakrf@gmail.com} 
\email[Corresponding author]{dimitris.ntal@gmail.com}

\date{\today}
\begin{document}

	\begin{abstract}
  		We prove that a domain in the Riemann sphere is Gromov hyperbolic if and only if it is conformally equivalent to a uniform circle domain. This resolves a conjecture of Bonk--Heinonen--Koskela and also verifies Koebe's conjecture (Kreisnormierungsproblem) for the class of Gromov hyperbolic domains. Moreover, the uniformizing conformal map from a Gromov hyperbolic domain onto a circle domain is unique up to M\"obius transformations. We also undertake a careful study of the geometry of inner uniform domains in the plane and prove the above uniformization and rigidity results for such domains. 
	\end{abstract}
\maketitle

\section{Introduction}

A geodesic metric space $X$ is called \textit{Gromov hyperbolic} if there exists $\delta>0$ such that, for each geodesic triangle, each side is within distance $\delta$ from the union of the other two sides. In other words, each geodesic triangle is \textit{$\delta$-thin}. In that case we say that $X$ is $\delta$-hyperbolic. Gromov \cite{Gromov:hyperbolic} showed that several fundamental characteristics of hyperbolic space can be recovered by that condition. Gromov hyperbolic spaces have been studied in depth and include, for example, all complete simply connected Riemannian manifolds of sectional curvature uniformly bounded above by a negative constant. See \cite{GhysHarpe:gromov} for further background.

Bonk, Heinonen, and Koskela \cite{BonkHeinonenKoskela:gromov_hyperbolic} studied Euclidean domains that are Gromov hyperbolic when equipped with the quasihyperbolic metric. Let $\Omega$ be a domain in the Riemann sphere $\widehat{\C}$ with $\partial \Omega\neq \emptyset$. We equip $\Omega$ with the spherical metric $\sigma$ on $\widehat{\C}$. The \textit{quasihyperbolic metric} in $\Omega$ is defined as
$$k_{\Omega}(x,y)= \inf_{\gamma} \int_\gamma \frac{1}{\dist_{\sigma} (z,\partial \Omega)} \frac{2|dz|}{1+|z|^2},$$ 
where the infimum is taken over all rectifiable paths $\gamma$ in $\Omega$ connecting $x$ and $y$. Recall that, $2(1+|z|^2)^{-1}|dz|$ is the spherical length element. We say that the domain $\Omega$ is Gromov hyperbolic if the metric space $(\Omega,k_{\Omega})$ is Gromov hyperbolic.

Let $\Omega\subset \widehat \C$ be a domain, equipped with the spherical metric. We say that $\Omega$ is a \textit{spherical uniform domain} if there exists a constant $A\geq1$ such that for every pair of points $x,y\in \Omega$ there exists a curve $\gamma\colon [0,1]\to \Omega$ such that $\gamma(0)=x$, $\gamma(1)=y$, 
\begin{align}\label{uniform:quasiconvex}
\ell_{\sigma} (\gamma)\le A\sigma(x,y),
\end{align}
and
\begin{align}\label{uniform:cigar}
\min\{\ell_{\sigma} (\gamma|_{[0,t]}), \ell_{\sigma}(\gamma|_{[t,1]} )\} \leq A \dist_{\sigma}(\gamma(t),\partial \Omega)
\end{align}
for all $t\in [0,1]$. If we replace condition   \eqref{uniform:quasiconvex} with the weaker condition that $\ell_{\sigma}(\gamma)\leq A\ell_{\sigma}(\gamma')$ for any path $\gamma'$ in $\Omega$ connecting $x$ and $y$ (so $\gamma$ has minimal length in a sense), then $\Omega$ is called a \textit{spherical inner uniform domain}. Euclidean (inner) uniform domains are defined in the obvious manner, using the Euclidean metric rather than the spherical. Uniform domains were introduced by Martio--Sarvas \cite{MartioSarvas:uniform} and Jones \cite{Jones:uniform} independently. From a point of view of quasiconformal geometry, these domains share several properties with the nicest possible domain, i.e., the unit disk, and its quasiconformal counterparts, i.e., quasidisks. We now summarize the main results of Bonk--Heinonen--Koskela for domains in the Riemann sphere.

\begin{theorem}[\cite{BonkHeinonenKoskela:gromov_hyperbolic}*{Theorems 1.11, 1.12, 3.6}]\label{theorem:bhk_inner}
Let $\Omega$ be a domain in the Riemann sphere. If $\Omega$ is a spherical inner uniform domain, then it is Gromov hyperbolic, quantitatively. Conversely, if $\Omega$ is Gromov hyperbolic, then it is conformally equivalent to a spherical inner uniform  domain.
\end{theorem}

We note that every simply connected domain in the Riemann sphere is Gromov hyperbolic, but not every simply connected domain is an inner uniform domain. Thus, the above theorem asserts that although the Euclidean geometry of a Gromov hyperbolic domain need not be nice, the domain can be transformed conformally to a domain with nice Euclidean geometry. A deep geometric characterization of Gromov hyperbolic domains was provided by Balogh--Buckley \cite{BaloghBuckley:gromov}.

Bonk--Heinonen--Koskela conjectured that Gromov hyperbolic domains are precisely the conformal images of uniform \textit{circle domains}; see \cite{BonkHeinonenKoskela:gromov_hyperbolic}*{p.~5}. A {circle domain} is a domain in the Riemann sphere with the property that each component of its boundary is a circle or a point. Our main theorem, in combination with Theorem \ref{theorem:bhk_inner} and the conformal invariance of hyperbolicity (see Theorem \ref{theorem:conformal_bilip}), provides an affirmative answer to the conjecture.

\begin{theorem}\label{theorem:main}
Let $\Omega$ be a domain in the Riemann sphere. If $\Omega$ is Gromov hyperbolic, then there exists a conformal map $f$ from $\Omega$ onto a circle domain $D$. Moreover, the conformal map $f$ is unique up to postcomposition with M\"obius transformations and the circle domain $D$ is a spherical uniform domain.
\end{theorem}

The last statement is quantitative in the sense that the constant of the uniform domain $D$ depends only on the hyperbolicity constant of $\Omega$ and on the radius of the largest spherical ball that is contained in $\Omega$. We are not aware whether the latter dependence can be dropped. We remark that the proof given in \cite{BonkHeinonenKoskela:gromov_hyperbolic} for the last statement of Theorem \ref{theorem:bhk_inner} is not quantitative, but our argument provides the same type of quantitative control for Theorem \ref{theorem:bhk_inner} as well. The proof of Theorem \ref{theorem:main} is given in Section \ref{section:uniformization_gh}.

Theorem \ref{theorem:main} provides an affirmative answer to \textit{Koebe's conjecture} or else \textit{
Kreis\-normierung\-sproblem} for the class of Gromov hyperbolic domains. Koebe's conjecture asserts that every domain in the Riemann sphere is conformally equivalent to a {circle domain} \cite{Koebe:Kreisnormierungsproblem}. The conjecture was established for finitely connected domains by Koebe \cite{Koebe:FiniteUniformization} and for countably connected domains by He--Schramm \cites{HeSchramm:Uniformization,Schramm:transboundary}, which is currently the best known result. The general case of the conjecture remains open. We refer to \cites{Rajala:koebe,NtalampekosRajala:exhaustion,EsmayliRajala:quasitripod} for some recent activity on the conjecture.

The last part of Theorem \ref{theorem:main} regarding uniqueness follows from the fact that a conformal map from a \textit{uniform} circle domain onto another circle domain is the restriction of a M\"obius transformation. In general, a circle domain is \textit{conformally rigid} if every conformal map onto another circle domain is the restriction of a M\"obius transformation. Not every circle domain is rigid and it is an open problem to characterize these domains. The rigidity of uniform circle domains, as in Theorem \ref{theorem:main}, follows from a recent result of Younsi and the second-named author \cite{NtalampekosYounsi:rigidity}. Other classes of rigid circle domains include finitely connected domains \cite{Koebe:FiniteUniformization}, countably connected domains \cite{HeSchramm:Uniformization}, and domains whose boundary has $\sigma$-finite length \cite{HeSchramm:Rigidity}. A more recent result of the second-named author \cite{Ntalampekos:rigidity_cned} demonstrates rigidity of a larger class of circle domains that includes all of the above classes.

Moreover, Theorem \ref{theorem:main}, in combination with  Theorem \ref{theorem:bhk_inner} and the conformal invariance of hyperbolicity (see Theorem \ref{theorem:conformal_bilip}), provides a characterization of domains that are conformally equivalent to uniform circle domains. It would be interesting to find a characterization for more general classes of circle domains, such as \textit{John domains}, i.e., domains any two points of which can be connected by a curve satisfying \eqref{uniform:cigar}.

To prove Theorem \ref{theorem:main}, we first use Theorem \ref{theorem:bhk_inner} to reduce to the case of spherical inner uniform domains; then we also reduce to the case of Euclidean inner uniform domains. In Section \ref{section:inner_euclidean} we establish several geometric properties of such domains. The most technical results of the paper are summarized in the next theorem. 

\begin{theorem}
Let $\Omega\subset \C$ be a Euclidean inner uniform domain. The following statements hold quantitatively.
\begin{enumerate}[label=\normalfont(\arabic*)]
	\item The complementary components of $\Omega$ are uniformly relatively separated.
	\item The complementary components of $\Omega$ have uniformly bounded turning.
	\item If $\Omega$ contains a neighborhood of $\infty$, there exists a sequence $\{\Omega_n\}_{n\in \N}$ of finitely connected domains that are inner uniform with the same constant such that $\Omega\subset \Omega_{n+1}\subset \Omega_n$ for each $n\in \N$ and $(\bigcap_{n\in \N}\Omega_n) \setminus \Omega$ is a totally disconnected set.
\end{enumerate}
\end{theorem}

See Theorem \ref{theorem:relative_distance}, Theorem \ref{theorem:bounded_turning}, and Corollary \ref{corollary:approximation} for more precise statements. Using these geometric properties in Section \ref{section:uniformization} we establish the uniformization of Euclidean inner uniform domains.

\begin{theorem}\label{theorem:uniformization_euclidean}
Let $\Omega\subset \C$ be a Euclidean inner uniform domain that contains a neighborhood of $\infty$. Then there exists a conformal map $f$ from $\Omega$ onto a circle domain $D\subset \C$ that contains a neighborhood of $\infty$ and $f$ extends continuously to $\infty$ so that $f(\infty)=\infty$. Moreover, the conformal map $f$ is unique up to postcomposition with conformal automorphisms of $\C$  and the circle domain $D$ is a Euclidean uniform domain, quantitatively.
\end{theorem}

As an application, we give a quasisymmetric characterization of inner uniformity. 

\begin{theorem}\label{theorem:qs_characterization}
Let $\Omega\subset \C$ be a domain that contains a neighborhood of $\infty$ and denote by $\rho_{\Omega}$ the inner diameter metric in $\Omega$. Then $\Omega$ is a Euclidean inner uniform domain if and only if $(\Omega,\rho_\Omega)$ is quasisymmetrically equivalent to a Euclidean uniform circle domain in $\C$. The statement is quantitative.
\end{theorem}

See Section \ref{section:inner_euclidean} for the definition of the inner diameter metric $\rho_{\Omega}$. Both above theorems are quantitative in the usual sense; namely the parameters in the conclusions depend only on the parameters in the assumptions and not on the particular domain $\Omega$ or map $f$. For versions of the above theorems for bounded domains, see Theorem \ref{theorem:bounded_uniformization} and Remark \ref{remark:qs_bounded}.

The equivalence in Theorem \ref{theorem:qs_characterization} was already known in the case of simply connected domains (see \cite{Heinonen:John}*{Theorem 3.1}, \cite{Vaisala:cylindrical}*{Theorem 2.20}, \cite{NakkiVaisala:john}*{Theorems 3.6, 3.9}, and Lemma \ref{lemma:qs_uniform_curves} below). Namely, the following are equivalent for a bounded simply connected domain $\Omega\subset \C$.
\begin{enumerate}
\item $\Omega$ is {John domain}.
\item Every conformal map $f\colon \D\to (\Omega,\rho_{\Omega})$ is quasisymmetric.
\item $\Omega$ is an inner uniform domain.
\end{enumerate}
However, for infinitely connected domains, inner uniformity is strictly stronger than the John property \cite{Vaisala:inner}*{Example 2.18}.

\section{Euclidean inner uniform domains}\label{section:inner_euclidean}

If $a$ is a parameter, we use the notation $C(a)$ for a positive constant that depends only on the parameter $a$. 
A \textit{curve} in a metric space $X$ is a continuous map $\gamma$ from a compact interval into $X$. The \textit{trace} of a curve $\gamma$ is its image and is denoted by $|\gamma|$. We define $\ell(\gamma)$ to be the length of $\gamma$ and $d(\gamma)$ to be the diameter of $|\gamma|$.

In this section we use exclusively the Euclidean metric on subsets of the plane. Let $\Omega\subset \R^2$ be a domain. For $a,b\in \Omega$ the \textit{inner metric} $\lambda_{\Omega}(a,b)$ is defined as the infimum of the length $\ell(\gamma)$ over all paths $\gamma\colon [0,1]\to \Omega$ such that $\gamma(0)=a$ and $\gamma(1)=b$. The \textit{inner diameter metric} $\rho_{\Omega}(a,b)$ is defined as the infimum of the diameter $d(\gamma)= \diam \gamma([0,1])$ over all paths $\gamma$ as above. For all $a,b\in \Omega$, we have
$$|a-b|\leq \rho_{\Omega}(a,b)\leq \lambda_{\Omega}(a,b).$$

Let $A\geq 1$. We say that a curve $\gamma\colon [0,1]\to \Omega$ is $(A,\lambda_{\Omega})$-\textit{uniform} (with respect to $\Omega$) or equivalently \textit{inner $A$-uniform} (with respect to $\Omega$) if 
$$\ell(\gamma)\leq A \lambda_{\Omega}(\gamma(0),\gamma(1))\quad \textrm{and}\quad \min\{\ell (\gamma|_{[0,t]}), \ell(\gamma|_{[t,1]} )\} \leq A \dist(\gamma(t),\partial \Omega),$$
for all $t\in [0,1]$. We define the \textit{length $A$-cigar of} $\gamma$ to be
$$\cig_{\ell}(\gamma, A)=\bigcup_{t\in [0,1]} B(\gamma(t), A^{-1}\min\{\ell (\gamma|_{[0,t]}), \ell(\gamma|_{[t,1]} )\}).$$
So, the second condition in the definition of an inner uniform curve is equivalent to the condition that $\cig_\ell(\gamma,A)\subset \Omega$. A domain $\Omega\subset \R^2$ is \textit{inner $A$-uniform} if for every pair of points $a,b\in \Omega$ there exists an inner $A$-uniform curve $\gamma\colon [0,1]\to \Omega$ with $\gamma(0)=a$ and $\gamma(1)=b$. To emphasize the metric, in that case we say that $\Omega$ is $(A,\lambda_{\Omega})$-uniform. We say that $\Omega$ is an inner uniform domain if it is inner $A$-uniform for some $A\geq 1$.

We say that a curve $\gamma\colon [0,1]\to \Omega$ is $(A,\rho_{\Omega})$\textit{-uniform} (with respect to $\Omega$) if 
$$d(\gamma)\leq A \rho_{\Omega}(\gamma(0),\gamma(1))\quad \textrm{and}\quad \min\{d (\gamma|_{[0,t]}), d(\gamma|_{[t,1]} )\} \leq A \dist(\gamma(t),\partial \Omega)$$
for all $t\in [0,1]$. The second condition is equivalent to requiring that $\cig_d(\gamma,A)\subset \Omega$, where the \textit{diameter $A$-cigar of $\gamma$} is defined by
$$\cig_{d}(\gamma, A)=\bigcup_{t\in [0,1]} B(\gamma(t), A^{-1}\min\{d(\gamma|_{[0,t]}), d(\gamma|_{[t,1]} )\}).$$
When we say that an arbitrary curve $\gamma$, is an $(A,\rho_{\Omega})$-uniform curve, it is implicitly understood that $\gamma$ is an $(A,\rho_{\Omega})$-uniform curve with respect to the domain $\Omega$ that appears a subscript in $\rho_{\Omega}$. We will use an equivalent definition for inner uniform domains in Euclidean space due to V\"ais\"al\"a, which allows us to replace lengths with diameters.
\begin{theorem}[\cite{Vaisala:inner}*{Theorem 3.11}]\label{theorem:vaisala}
Let $\Omega\subset \R^2$ be a domain. The following are quantitatively equivalent.
\begin{enumerate}[label=\normalfont(\arabic*)]
	\item $\Omega$ is an inner uniform domain.
	\item There exists $A\geq 1$ such that for every pair of points $a,b\in \Omega$ there exists an $(A,\rho_{\Omega})$-inner uniform curve $\gamma\colon [0,1]\to \Omega$ with $\gamma(0)=a$, $\gamma(1)=b$.
\end{enumerate}
\end{theorem}

For $A\geq 1$ a curve $\gamma\colon [0,1]\to \Omega$ is called an $A$\textit{-uniform curve} if 
$$\ell(\gamma)\leq A|\gamma(0)-\gamma(1)| \quad \textrm{and}\quad \cig_{\ell}(\gamma,A)\subset \Omega.$$
A domain $\Omega\subset \R^2$ is \textit{uniform} if there exists $A\geq 1$ such that for every pair of points $a,b\in \Omega$ there exists an $A$-uniform curve $\gamma\colon [0,1]\to \Omega$ with $\gamma(0)=a$ and $\gamma(1)=b$. In that case we say that $\Omega$ is an $A$-uniform domain.

\subsection{Geometry of inner uniform domains}\label{section:inner_geometry}
We establish two fundamental geometric properties of inner uniform domains.

\begin{theorem}[Separation of boundary components]
\label{theorem:relative_distance}
Let $A\geq 1$ and $\Omega\subset \R^2$ be an inner $A$-uniform domain. Then for every pair of components $S_1,S_2$ of $\R^2\setminus \Omega$ we have
$$\min \{\diam S_1,\diam S_2\}\leq C(A) \dist(S_1,S_2).$$
In particular, there exists at most one unbounded component of $\R^2\setminus \Omega$.
\end{theorem}
The conclusion is special to two dimensions and is not true in higher dimensions even in uniform domains. A proof of that result in the case of uniform domains was provided by Gehring \cite{Gehring:injectivity}*{Lemma 3} and relies on the fact that uniform domains are \textit{annularly linearly connected}, a property that fails in inner uniform domains. 

\begin{theorem}[Bounded turning of complement]\label{theorem:bounded_turning}
Let $A\geq 1$ and $\Omega\subset \R^2$ be an inner $A$-uniform domain. Then each component $S$ of $\R^2\setminus \Omega$ has $C(A)$-bounded turning.
\end{theorem}
Here we say that a metric space $(X,d_X)$ has \textit{bounded turning} if there exists a constant $L\geq 1$ such that for each pair $x,y\in X$ there exists a connected set $E\subset X$ that contains $x$ and $y$ with $\diam E\leq  L d_X(x,y)$. In this case we say that $X$ has $L$-bounded turning.  If, in addition, $X$ is locally compact, then one may take $E$ to be a path; see Lemma \ref{lemma:bt_definitions}. The conclusion of Theorem \ref{theorem:bounded_turning} is not true for $\partial S$ rather than $S$, as one can see by taking $\Omega$ be the complement of a long rectangle.

The key in the proof of Theorem \ref{theorem:relative_distance} and Theorem \ref{theorem:bounded_turning} is the next lemma. 

\begin{lemma}\label{lemma:relative_distance}
Let $A\geq 1$ and $\Omega\subset \R^2$ be an inner $A$-uniform domain. Let $z_0\in \R^2$, $R>0$, and consider a pair of components $S_1,S_2$ of $\br B(z_0,R)\cap (\R^2\setminus \Omega)$. Then
$$\min \{\diam S_1,\diam S_2\}\leq C(A)\max\{ \dist(z_0,S_1),\dist(z_0,S_2)\}. $$
\end{lemma}

In the proof of the lemma we will use an important fact about planar topology, known as Zoretti's theorem. Below $N_{\varepsilon}(S)$ denotes the open $\varepsilon$-neighborhood of $S$.

\begin{theorem}[Zoretti's theorem; \cite{Whyburn:topology}*{Corollary VI.3.11}]\label{theorem:zoretti}
Let $E\subset \R^2$ be a compact set and $S$ be a component of $E$. For each $\varepsilon>0$ there exists a Jordan region $U$ such that $S\subset U$, $\partial U\cap E=\emptyset$, and $\partial U\subset N_{\varepsilon }(S)$.
\end{theorem}

\begin{proof}[Proof of Lemma \ref{lemma:relative_distance}]
By Theorem \ref{theorem:vaisala}, we may assume that $\Omega$ is an $(A,\rho_{\Omega})$-uniform domain for some $A\geq 1$. Let $S_1,S_2$ be non-degenerate components of $\br B(z_0,R)\cap (\R^2\setminus \Omega)$.  Consider points $z_1\in S_1$ and $z_2\in S_2$ such that $\dist(z_0,S_i)=|z_0-z_i|$ for $i=1,2$ and let $L$ be the union of the line segments $[z_0,z_1]$ and $[z_0,z_2]$. Note that for each $z\in  L$ we have $\dist(z,\partial \Omega)\leq \max\{|z_0-z_1|,|z_0-z_2|\}\eqqcolon r$. Let $H>(A+1)^2$ and suppose that 
$$\diam S_i>2Hr$$
for $i=1,2$. We define $p= \frac{H}{A+1}$ and observe that $p>1$ and $p+1<H$. Therefore, $2R\geq \diam S_1>2Hr >2(p+1)r$, so $\br B(z_0,pr) \subset B(z_0,R)$.

By Zoretti's theorem (Theorem \ref{theorem:zoretti}) there exists a Jordan curve $J\subset \Omega\cup (\R^2\setminus \br B(z_0,R))$ (i.e., the complement of $\br B(z_0,R)\cap (\R^2\setminus \Omega)$) separating $S_1$ from $S_2$. Since the set $L$ connects $S_1$ and $S_2$, we see that $J\cap L\neq \emptyset$. Note that any curve that is homotopic to $J$ in $\R^2\setminus (S_1\cup S_2)$ also separates $S_1$ and $S_2$ (because the winding number is invariant under homotopies \cite{Burckel:complex}*{Theorem IV.4.12}) and therefore intersects the set $L$. We have $\diam J\geq \diam S_i > 2(p+1)r$ for some $i\in \{1,2\}$, so $J\not\subset B(z_0,pr)$.

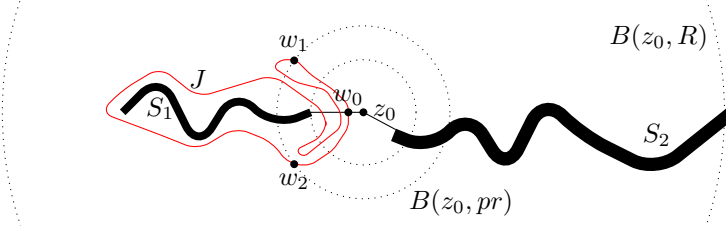
\begin{figure}
\centering
\begin{tikzpicture}
	\clip (-4.82,-1.5) rectangle (5.02,1.5);
	\draw[dotted] (0.2,0) circle (4.8);	
	\node[left] at (4.9,1) {$B(z_0,R)$};
	
	\draw (-1/2,0)--(0.2,0);
	\draw[dotted] (0.2,0) circle (0.7);
	\draw[dotted] (0.2,0) circle (1.15);
	\node at (1.5,-1.2) {$B(z_0,pr)$};
	
	\draw[rounded corners=10pt, line width=3pt] (-3,0) --(-2.5,.5) node[yshift=-0.45cm]{$S_1$} --(-2,-.5)--(-1.5,0.3)--(-1,-0.2)--(-0.5,0);
	
	\draw[red, rounded corners=2pt] (-3.3,0)[rounded corners=5pt] --(-3,0.3)--( -2.5,0.6)-- (-2,.2)node[black,above]{$J$}-- (-1,0.5)--(-0.3,0)--(-0.3,-0.3)[rounded corners=2pt]--(-0.7,-0.5)--(-0.6,-0.6)--(-0.1,-0.2)--(-0.1,0)--(-0.3,0.3)--(-0.7,0.4)--(-1,0.6)--(-0.9,0.7)--(-0.7,0.7)--(-0.5,0.5)--(-0.2,0.3)--(0,0.1)--(0,-0.2)--(-0.2,-0.5)--(-0.5,-0.7)[rounded corners=5pt]--(-0.8,-0.7)--(-1,-0.4)--(-1.5,-0.2)--(-2,-0.5)--(-2.5,-0.3)--cycle;
	
	\fill[black] (-0.72,0.69) circle (1.5pt) node[above] {$w_1$};
	\fill[black] (-0.72,-0.69) circle (1.5pt) node[below] {$w_2$};
	\fill[black] (0,0) circle (1.5pt) node[xshift=0cm, yshift=0.2cm] {$w_0$};	
	\fill[black](0.2,0) circle (1.5pt);
		
	\draw (0.2,0) node[right] {$z_0$}--(0.63,-0.23);		
	\draw[rounded corners=10pt, line width=5pt,xshift=0.1cm, yshift=-0.3cm] (5,0.3) node[shift={(-1,-0.3)}] {$S_2$}--(4,-0.5)--(3,0)--(2.5,.5)--(2,-.5)--(1.5,0.3)--(1,-0.2)--(0.5,0); 
\end{tikzpicture}
	\caption{The curve $J$ that separates $S_1$ from $S_2$.}\label{fig:components}
\end{figure}

Since $L\subset \br B(z_0,r)\subset B(z_0,pr)$ and $J\cap L\neq \emptyset$, there exists at least one component of the set $J\cap  B(z_0,pr)$ that intersects $L$. Consider the family of components $J_i$, $i\in I$, of $J\cap B(z_0,pr)$ that intersect $L$. The fact that $J\not\subset B(z_0,pr)$ and \cite{Whyburn:topology}*{I.10.2} imply that each set $J_i$ is an open arc with endpoints on $\partial B(z_0,pr)$. Moreover, since $L\subset \br B(z_0,r)$, we have $\diam J_i\geq (p-1)r$. The fact that $J$ is homeomorphic to the unit circle implies that $\{J_i\}_{i\in I}$ is a finite collection. We claim that there exists at least one component $J_i$ such that every curve that is homotopic to $J_i$ in $\R^2\setminus (S_1\cup S_2)$ rel.\ to the endpoints  must intersect $L$; see Figure \ref{fig:components}. Suppose instead that for each $i\in I$ the curve $J_i$ is homotopic  in $\R^2\setminus (S_1\cup S_2)$ rel.\ to the endpoints to a curve that is disjoint from $L$. By pasting these homotopies, one can then construct a homotopy in $\R^2\setminus (S_1\cup S_2)$ from $J$ to a curve that is disjoint from $L$. This is a contradiction. 

Let $\gamma$ be a parametrization of a component of $J\cap B(z_0,pr)$ with the above property. Let $w_1,w_2\in \partial B(z_0,pr)$ be the endpoints of $\gamma$. We have 
$$\rho_{\Omega}(w_1,w_2)\leq d(\gamma) \leq 2p r.$$ 
Note that $J\cap \partial B(z_0,pr)\subset \Omega$ because $\br B(z_0,pr)\subset B(z_0,R)$. Thus, $w_1,w_2\in \Omega$. Consider an $(A,\rho_{\Omega})$-uniform curve $\alpha$ in $\Omega$ connecting $w_1$ and $w_2$. We form a loop $\beta$ by concatenating $\alpha$ and $\gamma$. Suppose that $\beta$ is not null-homotopic in $\R^2\setminus (S_1\cup S_2)$. This implies that $\beta$ has non-zero winding number around all points of $S_i$ for some $i\in \{1,2\}$ (see \cite{Burckel:complex}*{Theorem IV.4.12}). As a consequence, the set $S_i$ lies in a bounded component of $\R^2\setminus |\beta|$ (see \cite{Burckel:complex}*{Corollary IV.4.3}). Thus,  
$$2Hr < \diam S_i \leq  d(\alpha)+d(\gamma)\leq  A\rho_{\Omega}(w_1,w_2)+d(\gamma)\leq 2(A+1)pr =2 Hr.$$
This is a contradiction. We conclude that $\beta$ is null-homotopic in $\R^2\setminus (S_1\cup S_2)$. Thus, $\gamma$ is homotopic to $\alpha$ in $\R^2\setminus (S_1\cup S_2)$ rel.\ to the endpoints. By the choice of $\gamma$, this implies that $\alpha$ intersects the set $L$ at a point $w_0$; see Figure \ref{fig:components}. Since $\alpha$ is an $(A,\rho_{\Omega})$-uniform curve, we have 
$$(p-1)r \leq \min \{ |w_0-w_1|,|w_0-w_2|\} \leq A \dist(w_0,\partial \Omega)\leq Ar.$$
Hence, $p-1\leq A$, i.e., $H\leq (A+1)^2$. This is a contradiction. Therefore, for every $H>(A+1)^2$ we have $\min\{\diam S_1,\diam S_2\}\leq 2Hr$. This proves the statement in the lemma with $C(A)=2(A+1)^2$.
\end{proof}

\begin{proof}[Proof of Theorem \ref{theorem:relative_distance}]
Let $S_1,S_2$ be non-degenerate components of $\R^2\setminus \Omega$. Consider points $z_1\in S_1$ and $z_2\in S_2$ such that $|z_1-z_2|=\dist(S_1,S_2)\eqqcolon 2r$ and let $z_0=(z_1+z_2)/2$. If $S_1$ and $S_2$ are unbounded, then for each $R\geq r$ and for $i=1,2$ there exists a component $S_i'(R)$ of $S_i\cap \br B(z_0,R)$ such that $S_i'(R)\cap \partial B(z_0,R)\neq \emptyset$ and $z_i\in S_i'(R)$; see \cite{Whyburn:topology}*{I.10.1}. By Lemma \ref{lemma:relative_distance} we have
$$R-r\leq  \min\{\diam S_1'(R), \diam S_2'(R)\}\leq C(A)r$$
for every $R\geq r$. This is a contradiction. Thus, one of $S_1,S_2$ is bounded. 

Suppose that $S_1$ is bounded and $S_2$ is unbounded. There exists $R_1\geq r$ such that $S_1\subset \br B(z_0,R_1)$ and $S_1\cap \partial B(z_0,R_1)\neq \emptyset$. Then $R_1-r\leq \diam S_1\leq 2R_1$. Let $S_2'$ be the component of $S_2\cap \br B(z_0,R_1)$ that contains $z_2$. Then $\diam S_2' \geq R_1-r$. 
By Lemma \ref{lemma:relative_distance} we have
$$R_1-r \leq   \min\{\diam S_1, \diam S_2'\} \leq C(A) r.$$
Therefore, 
$$\diam S_1 \leq 2R_1 \leq 2(C(A)+1)r= (C(A)+1)\dist(S_1,S_2).$$

Finally, suppose that both $S_1$ and $S_2$ are bounded. Let $R>0$ such that $S_1\cup S_2\subset B(z_0,R)$. By Lemma \ref{lemma:relative_distance} we obtain immediately the desired conclusion.
\end{proof}

\begin{corollary}\label{corollary:components}
Let $A\geq 1$ and $\Omega\subset \R^2$ be an inner $A$-uniform domain. Then for every $r,R>0$ the number of components $S$ of $\R^2\setminus \Omega$ with $S\cap B(0,R)\neq \emptyset$ and $\diam S>r$ is bounded above by $C(A)(1+R^2/r^2)$ and in particular is finite.
\end{corollary}
\begin{proof}
Let $N\in \N$ and  $S_1,\dots, S_N$ be components of $\R^2\setminus \Omega$ with $S_i\cap B(0,R)\neq \emptyset$ and $\diam S_i>r$ for $i\in \{1,\dots,N\}$. By Theorem \ref{theorem:relative_distance} we have $\dist(S_i,S_j) \geq C(A)r$ for $i\neq j$. Let $x_i\in S_i\cap B(0,R)$ and $B_i=B(x_i,C(A)r/2)$, $i\in \{1,\dots,N\}$. The balls $B_i$, $i\in \{1,\dots,N\}$, are pairwise disjoint and thus the area of their union is equal to $N\pi C(A)^2 r^2/4$. The union is contained in $B(0,R+C(A)r/2)$, whose area is bounded above by $C_1(A)(R^2+r^2)$. Therefore, $N\leq C_2(A)(1+R^2/r^2)$.
\end{proof}

Now, we turn to the proof of Theorem \ref{theorem:bounded_turning}.

\begin{proof}[Proof of Theorem \ref{theorem:bounded_turning}]
Let $H>1$ to be determined. Let $z_1,z_2\in S$ and set $z_0=(z_1+z_2)/2$ and $r=|z_1-z_2|/2$. 
We will show that $z_1,z_2$ lie in the same component of $\br B(z_0,Hr)\cap S$, upon choosing $H$ to be sufficiently large, depending on $A$. Suppose that there exist distinct components $S_1,S_2$ of $\br B(z_0,Hr)\cap  S$ such that $z_i\in S_i$, $i=1,2$. In particular, $S_i$ intersects $\partial B(z_0,Hr)$ (see \cite{Whyburn:topology}*{I.10.1}), so $\diam S_i\geq Hr-r$ for $i=1,2$. By Lemma \ref{lemma:relative_distance}, we have
$$ Hr-r \leq \min\{ \diam S_1,\diam S_2\}\leq C(A)r.$$
Thus, $H\leq C(A)+1$. If we consider $H> C(A)+1$ in the beginning of the proof, then we obtain a contradiction.
\end{proof}

We also record a lemma for future reference. 

\begin{lemma}\label{lemma:bt_definitions}
Let $(X,d_X)$ be a locally compact metric space.
\begin{enumerate}[label=\normalfont (\arabic*)]
	\item If $(X,d_X)$ is locally connected and $E\subset X$ is a connected set then, for every $\varepsilon>0$, any two points $x,y\in E$ can be connected by a path $\gamma$ with $d(\gamma)\leq \diam(E)+\varepsilon$.
	\item  If $(X,d_X)$ has $L$-bounded turning for some $L\geq 1$, then for every $L'>L$, any two points $x,y\in X$ can be connected by a path $\gamma$ with $d(\gamma)\leq L' d_X(x,y)$.
\end{enumerate}

\end{lemma}
\begin{proof}
Let $\varepsilon>0$ and $Y$ be the connected component of $N_{\varepsilon/2}(E)$ that contains $E$. Then $\diam(Y)\leq  \diam(E)+\varepsilon$. Also, $Y$ is a locally compact, locally connected, and connected space. By \cite{Whyburn:topology}*{II.5.2, p.~38}, every pair of points $x,y\in Y$ can be connected by a path $\gamma$ in $Y$, so $d(\gamma)\leq \diam E+\varepsilon$. This completes the proof of the first part. The second part follows from the first part upon observing that a space of bounded turning is locally connected. 
\end{proof}

\subsection{Approximation of inner uniform domains}

The next theorem is the main result of this section. Its proof relies an all results from Section \ref{section:inner_geometry}.

\begin{theorem}\label{theorem:approximation}
Let $A\geq 1$ and $\Omega\subset \R^2$ be an inner $A$-uniform domain. For $\delta>0$ let $D\subset \R^2$ be the domain such that $\R^2\setminus D$ is the union of the components $S$ of $\R^2\setminus \Omega$ with $\diam S >\delta$. Then $D$ is an inner $C(A)$-uniform domain.
\end{theorem}
Most importantly, $D$ is inner uniform with constant not depending on $\delta$. Note that all components of $\R^2\setminus D$ are necessarily larger than all other components of $\R^2\setminus \Omega$. This condition is crucial and the conclusion of the theorem is not true in general for domains $D\supset \Omega$ whose complement consists of an arbitrary collection of components of $\R^2\setminus \Omega$. For example, let $\Omega$ be a domain whose complement consists of $E=[0,1]\times\{0\}$ and a countable collection $F$ of points in $\R^2\setminus E$ whose closure contains $E$. One can choose the points of $F$ sparsely so that $\Omega$ is an inner uniform domain. If $D$ is a domain whose complement consists of all points in $F$ outside a small neighborhood $U$ of $E$, then $D$ is not inner uniform with uniform constants independent of the choice of $U$.

We establish some preliminary statements.

\begin{lemma}[Avoidance of boundary components]\label{lemma:path_omega}
Let $\Omega\subset \R^2$ be a domain and $r>0$. Suppose that $\gamma$ is a curve in $\R^2$ connecting two points $a,b\in \Omega$ such that each component $S$ of $\R^2\setminus \Omega$ that intersects $\gamma$ satisfies $\diam S < r$. Then there exists a curve $\widetilde \gamma$ in $\Omega$ connecting $a$ and $b$ such that $|\widetilde \gamma|\subset N_r(|\gamma|)$ and $d(\widetilde \gamma)< 2r+d(\gamma)$.
\end{lemma}
\begin{proof}
By assumption, each component $S$ of $\R^2\setminus \Omega$ that intersects $\gamma$ is bounded. By Zoretti's theorem (Theorem \ref{theorem:zoretti}),  there exists a Jordan region $U_S$ containing $S$ whose boundary $J_S$ is contained in $\Omega$, $J_S$ separates $S$ from $\{a,b\}$, and $\diam J_S<r$. Since $|\gamma|\cap (\R^2\setminus \Omega)$ is a compact set, it can be covered by finitely many Jordan regions $U_{S_1},\dots,U_{S_N}$ as above.  One can show inductively that there exists a path $\widetilde \gamma$ that connects $a,b$, has trace in $|\gamma|\cup \bigcup_{i=1}^N J_{S_i} \subset N_r(|\gamma|)$, and does not intersect $\bigcup_{i=1}^N (U_{S_i}\setminus \Omega)$. Therefore, $\widetilde \gamma$ is a path in $\Omega$ and $d(\widetilde \gamma)< 2r+d(\gamma)$. 
\end{proof}

\begin{lemma}[Concatenation of cigars]\label{lemma:concatenation_cigar}
Let $D\subset \R^2$ be a domain and $\gamma_1,\gamma_2$ be curves in $D$ with a common endpoint $x_0$ such that 
$$\dist(x_0,\partial D)\geq A^{-1}\min\{d(\gamma_1),d(\gamma_2)\}\,\,\, \textrm{and} \,\,\, \cig_d(\gamma_i,A)\subset D \,\,\,\textrm{for $i=1,2$}$$
for some $A\geq 1$. If $\gamma$ is the concatenation of $\gamma_1$ and $\gamma_2$ at $x_0$, then $\cig_d(\gamma, C(A)) \subset D$. The same statement is true if the diameter $d$ in the assumptions and conclusion is replaced with the length $\ell$.
\end{lemma}
\begin{proof}
We present the proof in the case of $d$; there is no change if one considers $\ell$. Without loss of generality, $d(\gamma_1)\leq d(\gamma_2)$. Define $r_0=\dist(x_0,\partial D)$. Consider a parametrization $\gamma\colon [0,1]\to D$ such that $\gamma(1/2)=x_0$, $\gamma|_{[0,1/2]}$ is the curve $\gamma_1$, and $\gamma|_{[1/2,1]}$ is the curve $\gamma_2$ up to reparametrization.  Let $t\in [0,1/2]$. If $\gamma(t)\in  B(x_0,r_0/2)$, then 
$$\min\{d(\gamma|_{[0,t]}) , d(\gamma|_{[t,1]}) \}\leq d(\gamma|_{[0,t]}) \leq d(\gamma_1)\leq Ar_0\leq 2A\dist(\gamma(t),\partial D).$$
If $\gamma(t)\notin B(x_0,r_0/2)$, then 
$$d(\gamma|_{[t,1/2]})\geq \frac{r_0}{2} \geq \frac{1}{2A} d(\gamma_1)\geq \frac{1}{2A} d( \gamma|_{[0,t]}).$$
Since $\cig_d(\gamma_1,A)\subset D$, we have 
\begin{align*}
\min\{d(\gamma|_{[0,t]}) , d(\gamma|_{[t,1]}) \}&\leq d(\gamma|_{[0,t]}) \leq 2A \min \{d(\gamma|_{[0,t]}), d(\gamma|_{[t,1/2]}) \}\\ 
&\leq 2A^2 \dist(\gamma(t),\partial D).
\end{align*}
Next, let $t\in [1/2,1]$.  If $\gamma(t)\in  B(x_0,r_0/
2)$, since $\cig_d(\gamma_2,A)\subset D$, we have
\begin{align*}
\min\{d(\gamma|_{[0,t]}) , d(\gamma|_{[t,1]}) \}&\leq \min\{ d(\gamma_1)+d(\gamma|_{[1/2,t]}),d(\gamma|_{[t,1]}) \}\\
&\leq d(\gamma_1)+ \min\{ d(\gamma|_{[1/2,t]}),d(\gamma|_{[t,1]}) \}\\
&\leq Ar_0 + A\dist(\gamma(t), \partial D)\\
&\leq  3A\dist(\gamma(t), \partial D).
\end{align*}
Finally, suppose that $\gamma(t)\notin B(x_0,r_0/2)$, so
$$d(\gamma|_{[1/2,t]}) \geq \frac{r_0}{2} \geq \frac{1}{2A} d(\gamma_1).$$
Thus, 
$$d(\gamma|_{[0,t]}) \leq d(\gamma_1) + d(\gamma|_{[1/2,t]}) \leq (2A+1) d(\gamma|_{[1/2,t]}).$$
Since $\cig_d(\gamma_2,A)\subset D$, we conclude that 
\begin{align*}
\min\{d(\gamma|_{[0,t]}) , d(\gamma|_{[t,1]}) \}&\leq (2A+1)\min\{d(\gamma|_{[1/2,t]}) , d(\gamma|_{[t,1]}) \}\\
&\leq (2A+1) A \dist(\gamma(t),\partial D). 
\end{align*}
This completes the proof with $C(A)=(2A+1)A$.
\end{proof}

\begin{proof}[Proof of Theorem \ref{theorem:approximation}]
By Theorem \ref{theorem:vaisala}, we may assume that $\Omega$ is $(A,\rho_{\Omega})$-uniform for some $A\geq 1$. Let $D$ be a domain as in the statement. In particular, each component $S$ of $\R^2\setminus \Omega$ that is contained in $D$ satisfies $\diam S\leq \delta$.  We prove a preliminary statement.

\begin{claim}\label{theorem:approximation:claim}
Let $S$ be a component of $\R^2\setminus \Omega$ such that $S\subset D$. There exists a constant $M\geq 1$, depending on $A$, such that if $U$ is the open $(M^{-1}\diam S)$-neighborhood of $S$, then
\begin{align}\label{theorem:approximation:claim:first}
\dist(U,\R^2\setminus D)\geq  M^{-1}\diam S,
\end{align}
and moreover for every pair of points $x_1,x_2\in U$, there exists a path $\gamma \colon [0,1]\to U$ connecting $x_1,x_2$ such that  
\begin{enumerate}[label=\normalfont(\arabic*)]
	\item\label{cl1} $\dist(x_i,\partial D) \geq M^{-1}\max\{d(\gamma),\diam S\}$,
	\item\label{cl2} $d(\gamma)\leq M(\dist(x_1,S)+\dist(x_2,S)+|x_1-x_2|)$, and
	\item\label{cl3} $\cig_{d}(\gamma,M)\subset D.$
\end{enumerate}
\end{claim}
\begin{proof}[Proof of Claim \ref{theorem:approximation:claim}]
Let $T$ be a component of $\R^2\setminus D$. By the assumptions on the domain $D$ we have $\diam S\leq \delta<\diam T$. 
By Theorem \ref{theorem:relative_distance} there exists a constant $C_1(A)\geq 1$ such that
$$\dist(S,T)\geq C_1(A)^{-1}\min\{\diam S,\diam T\} = C_1(A)^{-1}\diam S.$$
Since this holds for all components $T$ of $\R^2\setminus D$, we obtain
\begin{align*}
\dist(S,\partial D)=\dist(S,\R^2\setminus D) \geq C_1(A)^{-1}\diam S.
\end{align*} 
Let $M_1=2C_1(A)$ and $M\geq M_1$. If $x\in U$, then $\dist(x,S)<M^{-1}\diam S$. This implies that $x\in D$ and 
\begin{align}\label{theorem:approximation:distances_lower}
\dist(x,\partial D)> (C_1(A)^{-1}-M^{-1})\diam S \geq  M_1^{-1}\diam S\geq M^{-1}\diam S.
\end{align}
This proves \eqref{theorem:approximation:claim:first}.

Next, let $x_1,x_2\in U$. For $i=1,2$ consider a line segment $L_i$ from $x_i$ to its closest point $x_i'\in S$. Then $L_i$ is contained in $U$. By Theorem \ref{theorem:bounded_turning} and Lemma \ref{lemma:bt_definitions}, there exists a path $\beta$ in $S$ connecting $x_1'$ and $x_2'$ such that 
$$d(\beta)\leq C_2(A) |x_1'-x_2'| \leq C_2(A)(\dist(x_1,S)+\dist(x_2,S) +|x_1-x_2|)$$
for some constant $C_2(A)\geq 1$. The concatenation $\gamma$ of the paths $L_1,L_2$, and $\beta$ satisfies
\begin{align*}
d(\gamma) &\leq 2C_2(A)(\dist(x_1,S)+\dist(x_2,S) +|x_1-x_2|).
\end{align*}
This verifies \ref{cl2}, provided that $M\geq 2C_2(A)$.
Moreover, since $|\gamma|\subset U$, by \eqref{theorem:approximation:distances_lower} we have 
$$d(\gamma)\leq  \diam U\leq (1+2M^{-1})\diam S \leq 3\diam S\leq 3M_1 \dist(x,\partial D)$$
for each point $x$ on $\gamma$. Thus, if we take $M\geq 3M_1$, then $\cig_d(\gamma, M)\subset D$ and \ref{cl3} is true. Also, the above inequality and \eqref{theorem:approximation:distances_lower} verify \ref{cl1}. Summarizing, the conclusions of the claim hold for $M=\max\{ 3M_1, 2C_2(A)\}$.
\end{proof}

\noindent
\textbf{Reductions.}  We assume for the moment that any two points of $\Omega$ can be connected by a $(C(A), \rho_D)$-uniform curve (with respect to the domain $D$). We will show that the same holds for any two points of $D$.

Let $a,b\in D$ be distinct points. Assume that $a$ lies in a component $S$ of $\R^2\setminus \Omega$. If $b\in S$, then by Claim \ref{theorem:approximation:claim} there exists an $(M,\rho_D)$-uniform curve connecting $a$ and $b$. Next, suppose that $b\notin S$.  Consider a curve $\gamma_D$ in $D$ connecting $a,b$ with $d(\gamma_D)\leq 2\rho_{D}(a,b)$. The curve $\gamma_D$ is not entirely contained in $S$. Hence, there exists a point $a'\in \Omega$ that lies on the curve $\gamma_D$ and satisfies 
$$\dist(a',S)<\min\{M^{-1}\diam S, d(\gamma_D)\}.$$
By Claim \ref{theorem:approximation:claim}, there exists a curve $\gamma_a$ connecting $a$ and $a'$ such that 
$$d(\gamma_a)\leq M(\dist(a',S)+ |a-a'|)\leq 2M d(\gamma_D) \quad \textrm{and} \quad \cig_d(\gamma_a,M)\subset D.$$
Similarly, if $b$ lies in another component of $\R^2\setminus \Omega$ there exists a curve $\gamma_b$ connecting $b$ to a point $b'\in \Omega$ that lies on the curve $\gamma_D$ and satisfies
$$d(\gamma_b)\leq  2M d(\gamma_D) \quad \textrm{and} \quad \cig_d(\gamma_b,M)\subset D.$$
If $b\in \Omega$, then $\gamma_b$ is taken to be a constant path and the above hold vacuously. 

By our assumption, there exists a $(C(A),\rho_D)$-uniform curve $\gamma'$ connecting the points $a',b'\in \Omega$. Then the concatenation $\gamma$ of $\gamma_a,\gamma',\gamma_b$ satisfies
\begin{align*}
d(\gamma)&\leq d(\gamma_a)+d(\gamma')+d(\gamma_b) \leq 4Md(\gamma_D)+ C(A)\rho_D(a',b')\\
&\leq (4M+C(A))d(\gamma_D) \leq 2(4M+C(A)) \rho_D(a,b).
\end{align*}
By Claim \ref{theorem:approximation:claim}, the common endpoint $a'$ of $\gamma_a$ and $\gamma'$ satisfies 
$$\dist(a',\partial D)\geq M^{-1}d(\gamma_a)$$
and an analogous inequality is true for $b'$. By applying Lemma \ref{lemma:concatenation_cigar} twice, we obtain $\cig_d(\gamma, C'(A))\subset D$.  Therefore, $\gamma$ is the desired curve.

\medskip
\noindent
\textbf{Main argument.}
It remains to show that if $a,b\in \Omega$ are distinct points, then there exists a $(C(A),\rho_D)$-uniform curve connecting them. Let $H>2$ to be determined, depending only on $A$. Suppose that $\rho_{\Omega}(a,b)\leq H\rho_D(a,b)$. Consider an $(A,\rho_{\Omega})$-uniform curve $\gamma$ connecting $a$ and $b$. Then $\cig_d(\gamma,A)\subset \Omega\subset D$ and $d(\gamma)\leq A\rho_{\Omega}(a,b)\leq AH\rho_D(a,b)$, so $\gamma$ is an $(AH, \rho_D)$-uniform curve. Next, suppose that 
\begin{align}\label{theorem:approximation:h}
\rho_{\Omega}(a,b)> H\rho_D(a,b).
\end{align}
Let $\gamma_D$ be a curve in $D$ connecting $a$ and $b$ such that $d(\gamma_D)\leq 2\rho_D(a,b)$. We claim that there exists a component $S$ of $\R^2\setminus \Omega$ that intersects $\gamma_D$ and
\begin{align}\label{theorem:approximation:s}
\diam S\geq \frac{H-2}{4}d(\gamma_D).
\end{align}
Suppose, instead that all components of $\R^2\setminus \Omega$ that intersect $\gamma_D$ satisfy the reverse inequality. By Lemma \ref{lemma:path_omega}, there exists a curve $\gamma_\Omega$ in $\Omega$ connecting $a$ and $b$ with 
$$d(\gamma_{\Omega})< 2\frac{H-2}{4}d(\gamma_D)+d(\gamma_D) = \frac{H}{2} d(\gamma_D).$$
Thus, 
$$\rho_\Omega(a,b)\leq d(\gamma_{\Omega}) < \frac{H}{2}d(\gamma_D) \leq H \rho_D(a,b).$$
This contradicts \eqref{theorem:approximation:h}. Thus, there exists a component $S$ of $\R^2\setminus \Omega$ that intersects $\gamma_D$ and satisfies \eqref{theorem:approximation:s}.

Suppose that there exist two such components $S_1$ and $S_2$. By Theorem \ref{theorem:relative_distance}, we have
$$d(\gamma_D)\geq \dist(S_1,S_2) \geq C_1(A)^{-1} \min\{\diam S_1,\diam S_2\}\geq C_1(A)^{-1}\frac{H-2}{4}d(\gamma_D).$$
If we require that $H>4C_1(A)+2$, we obtain a contradiction. Thus, there exists a unique component $S$ of $\R^2\setminus \Omega$ that intersects $\gamma_D$ and satisfies \eqref{theorem:approximation:h}.

Let $\gamma_a$ be a subpath of $\gamma_D$ from $a$ to a point $a_1\in \Omega$ so that $\gamma_a$ does not intersect $S$ and 
\begin{align}\label{theorem:approximation:a1}
\dist(a_1,S)< \min\{ M^{-1}\diam S, d(\gamma_D) \},
\end{align}
where $M$ is as in Claim \ref{theorem:approximation:claim}. Then all components $T$ of $\R^2\setminus \Omega$ that intersect $\gamma_a$ satisfy $\diam T< \frac{H-2}{4} d(\gamma_D) $. By Lemma \ref{lemma:path_omega}, we see that
$$\rho_\Omega(a,a_1)< 2\frac{H-2}{4}d(\gamma_D)+d(\gamma_D)= \frac{H}{2}d(\gamma_D).$$
Since $\Omega$ is inner uniform, there exists an $(A,\rho_{\Omega})$-uniform curve $\widetilde \gamma_a$ connecting $a$ and $a_1$. In particular, 
\begin{align}\label{theorem:approximation:a_tilde}
d(\widetilde \gamma_a) \leq A\rho_\Omega(a,a_1)\leq \frac{AH}{2}d(\gamma_D).
\end{align}
Similarly, there exists an $(A,\rho_{\Omega})$-uniform curve $\widetilde \gamma_b$  connecting $b$ to a point $b_1\in \Omega$ such that
\begin{align}\label{theorem:approximation:b1}
\dist(b_1,S)< \min\{ M^{-1}\diam S, d(\gamma_D) \} \,\,\,\textrm{and}\,\,\,d(\widetilde \gamma_b)\leq \frac{AH}{2}d(\gamma_D).
\end{align}

By Claim \ref{theorem:approximation:claim}, \eqref{theorem:approximation:a1}, and \eqref{theorem:approximation:b1}, there exists a path $\beta$ connecting $a_1,b_1$ such that
\begin{align}\label{theorem:approximation:beta_main}
d(\beta)\leq M( \dist(a_1,S)+\dist(b_1,S)+|a_1-b_1|)\leq 3Md(\gamma_D),
\end{align}
$\cig_d(\beta,M)\subset D$, and $\dist(a_1,\partial D)\geq M^{-1}\diam S$. The latter condition, combined with  \eqref{theorem:approximation:s}, \eqref{theorem:approximation:a_tilde}, and  \eqref{theorem:approximation:b1}, implies that 
$$\dist(a_1,\partial D) \geq M^{-1}\diam S \geq \frac{H-2}{4M}d(\gamma_D)\geq \frac{H-2}{2MAH}\max\{d(\widetilde \gamma_a),d(\widetilde \gamma_b)\};$$
the same holds for $b_1$ in place of $a_1$. Denote by $\gamma$ the concatenation of $\widetilde \gamma_a$, $\beta$, and $\widetilde \gamma_b$. By applying Lemma \ref{lemma:concatenation_cigar} twice, we obtain $\cig_d(\gamma, C_3(A)) \subset D$. Moreover, by \eqref{theorem:approximation:a_tilde},  \eqref{theorem:approximation:b1}, and \eqref{theorem:approximation:beta_main},
$$d(\gamma) \leq d(\widetilde \gamma_a)+d(\beta)+d(\widetilde \gamma_b)\leq (AH+3M)d(\gamma_D)\leq 2(AH+3M)\rho_D(a,b).$$
This completes the proof.
\end{proof}

We will need the following consequence of Theorem \ref{theorem:approximation}. Here, a \textit{neighborhood of $\infty$} is an open subset of the plane that contains the exterior of a ball.

\begin{corollary}[Approximation of inner uniform domains]\label{corollary:approximation}
Let $A\geq 1$ and $\Omega\subset {\R^2}$ be an inner $A$-uniform domain that is bounded or contains a neighborhood of $\infty$. There exists a sequence of finitely connected inner $C(A)$-uniform domains $\Omega_n$, $n\in \N$, such that 
\begin{enumerate}[label=\normalfont(\roman*)]
\item $\Omega \subset \Omega_{n+1}\subset \Omega_{n}$ for each $n\in \N$,
\item each non-degenerate component of $\R^2\setminus \Omega$ is a component of $\R^2\setminus \Omega_n$ for all sufficiently large $n\in \N$, and
\item the set $(\bigcap_{n=1}^\infty \Omega_n) \setminus \Omega$ is totally disconnected.  
\end{enumerate}
\end{corollary}

\begin{proof}
Since $\Omega$ is bounded or contains a neighborhood of $\infty$, all components of $\R^2\setminus \Omega$, with the exception of at most one component, are contained in a ball. By Corollary \ref{corollary:components}, for each $n\in \N$ there exist at most finitely many components $S$ of $\R^2\setminus \Omega$ with $\diam F>1/n$. For $n\in \N$, let $\Omega_n$ be the finitely connected domain whose boundary is the union of those components. By Theorem \ref{theorem:approximation}, $\Omega_n$ is an inner $C(A)$-uniform domain for each $n\in \N$. By construction, $\Omega_n\supset \Omega_{n+1}\supset \Omega$ for every $n\in \N$, and $(\bigcap_{n=1}^\infty \Omega_n) \setminus \Omega$ consists of the point boundary components of $\Omega$. 
\end{proof}

\subsection{Conformal invariance of uniformity}

A homeomorphism $f\colon (X,d_X)\to (Y,d_Y)$ between metric spaces is \textit{quasisymmetric} if there exists a homeomorphism $\eta\colon [0,\infty)\to [0,\infty)$ such that  for every triple of distinct points $x_i\in X$, $i=1,2,3$, and for $y_i=f(x_i)$, $i=1,2,3$, we have
\begin{align*}
\frac{d_Y(y_1,y_2)}{d_Y(y_1,y_3)}\leq \eta\left(\frac{d_X(x_1,x_2)}{d_X(x_1,x_3)}\right).
\end{align*}
In that case we say that $f$ is $\eta$-quasisymmetric. In general, a homeomorphism $\eta\colon [0,\infty)\to [0,\infty)$ as above is called a \textit{distortion function}. 

\begin{lemma}[Quasisymmetric invariance of uniform curves]\label{lemma:qs_uniform_curves}
Let $\Omega,D\subset \R^2$ be domains and $f\colon (\Omega,\rho_\Omega)\to (D,\rho_D)$ be an $\eta$-quasisymmetric map for some distortion function $\eta$. If $\gamma$ is an $(A,\rho_{\Omega})$-uniform curve for some $A\geq 1$, then $f\circ\gamma$ is an $(A',\rho_D)$-uniform curve for $A'=2\eta(2A)$.
\end{lemma}

\begin{proof}
Suppose that $\gamma\colon [0,1]\to \Omega$ is  an $(A,\rho_\Omega)$-uniform curve connecting two points $a,b\in \Omega$. We have $d(\gamma)\leq A\rho_{\Omega}(a,b)$. There exists a point $c$ on  $\gamma$ such that 
\begin{align*}
\frac{d(f\circ \gamma)}{ \rho_D(f(a),f(b))} \leq \frac{2\rho_D(f(a),f(c))}{\rho_D(f(a),f(b))} \leq 2 \eta \bigg( \frac{\rho_\Omega(a,c)}{\rho_{\Omega}(a,b)}\bigg) \leq 2\eta \left(\frac{d(\gamma)}{\rho_{\Omega}(a,b)}\right)\leq 2\eta(A).
\end{align*}
Next, we will show that $\cig_{d}(f\circ \gamma, 2\eta(2A))\subset D$. Let $t\in [0,1]$, $z=\gamma(t)$, and without loss of generality suppose that $p\coloneqq d(\gamma|_{[0,t]})\leq d(\gamma|_{[t,1]})$. We have $\cig_d(\gamma, 2A)\subset \cig_d(\gamma,A)\subset \Omega$ and in fact the closure of $\cig_d(\gamma, 2A)$ is a compact subset of $\Omega$. Let $w\in \Omega\setminus \cig_d(\gamma, 2A)$. This implies that $|z-w|\geq p(2A)^{-1}$. Therefore, for each $s\in [0,t]$ and for $x=\gamma(s)$ we have
$$|z-x|\leq p \leq 2A |z-w|.$$
We conclude that 
$$|f(z)-f(x)|\leq \eta(2A) |f(z)-f(w)|.$$
Since $s\in [0,t]$ is arbitrary and $f(w)\in  D\setminus f(\cig_d(\gamma, 2A))$ is arbitrary, we have
$$ d( f\circ \gamma|_{[0,t]})\leq 2\eta(2A) \dist(f(z), \partial D).$$
This shows that $\cig_d(f\circ \gamma, 2\eta(2A) )\subset D$, as desired. 
\end{proof}

A metric space $(X,d_X)$ is \textit{linearly locally connected} (LLC) if there exists a constant $M\geq 1$ such that for each ball $B_{d_X}(a,r)$ in $X$ the following two conditions hold.
\begin{enumerate}[label=\normalfont(LLC$_{\arabic*}$)]
\item\label{llc1} For every $x,y\in B_{d_X}(a,r)$ there exists a connected set $E\subset B_{d_X}(a,Mr)$ that contains $x$ and $y$.
\item\label{llc2} For every $x,y\in X\setminus B_{d_X}(a,r)$ there exists a connected set $E\subset X\setminus B_{d_X}(a,r/M)$ that contains $x$ and $y$.
\end{enumerate}
In that case, we say that $X$ is $M$-LLC.

\begin{lemma}\label{lemma:circle_llc}
Let $\Omega\subset \R^2$ be a circle domain. Then $\Omega$ is $1$-LLC.
\end{lemma}
\begin{proof}
Let $B$ be a ball with $\Omega\cap B\neq \emptyset$ and let $z,w\in \Omega\cap \br B$. We will show that there exists a path $\gamma$ connecting $z,w$ and lying in $\Omega\cap B$, except for the endpoints. This implies condition \ref{llc1} with constant $M=1$.

Let $z',w'\in \Omega\cap B$ be points so that the line segments $[z,z']$ and $[w,w']$ lie in $\Omega\cap \br B$. It suffices to show the claim for $z',w'$. Let $B'$ be a ball such that $z',w'\in B'\subset \br {B'}\subset B$ and let $0<\varepsilon<\dist(\partial B,\partial B')$. The collection of components $D_i$, $i\in I$, of $\R^2\setminus \Omega$ that intersect the line segment $[z',w']$ and satisfy $\diam D_i\geq \varepsilon$, $i\in I$, is finite. For each $i\in I$ we consider an arc $C_i$ of a circle concentric to $\partial D_i$ with slightly larger radius so that $C_i$ has its endpoints on $[z',w']$, $C_i\subset B'$, and $C_i$ does not intersect any component $S$ of $\R^2\setminus \Omega$ with $\diam S\geq \varepsilon$. Moreover, we may assume that $C_i\cap C_j=\emptyset$ for distinct $i,j\in I$. Let $\gamma$ be the path arising by replacing for each $i\in I$ the segment of $[z',w']$ between the endpoints of $C_i$ with $C_i$. By construction, $|\gamma|\subset B'$ and  every component of $\R^2\setminus \Omega$ that intersects $\gamma$ has diameter less than $\varepsilon$. By Lemma \ref{lemma:path_omega} there exists a path $\widetilde \gamma$ in $\Omega$ that connects $z',w'$ and $|\widetilde \gamma|\subset N_{\varepsilon}(|\gamma|)\subset N_{\varepsilon}(B') \subset B$. This completes the proof of the claim.

For condition \ref{llc2}, let $B$ be a ball with $\Omega \setminus B\neq \emptyset$. We apply an inversion so that a point $a\in B\setminus \partial \Omega$ is mapped to $\infty$ and $B$ is mapped to the exterior of a ball $B_1$. Note that the image of $\Omega\setminus \{a\}$ under the inversion is a circle domain $\Omega'\subset \R^2$. By the initial claim, any  two points in $\Omega'\cap \br {B_1}$ can be connected by a path in $\Omega'\cap \br {B_1}$. Thus, any two points in $\Omega\setminus  B$ can be connected by a path in $\Omega\setminus B$.
\end{proof}

\begin{theorem}[Conformal invariance of uniformity]\label{theorem:conformal_inner_uniform}
Let $\Omega,D\subset  \R^2$ be domains that contain a neighborhood of $\infty$ and  $f\colon \Omega\to D$ be a quasiconformal map that extends continuously to $\infty$ so that $f(\infty)=\infty$. If $\Omega$ is an inner uniform domain and $D$ is an {LLC} domain, then $f\colon (\Omega,\rho_\Omega)\to D$ is quasisymmetric and $D$ is a uniform domain, quantitatively. 
\end{theorem}

The proof relies on several notions from analysis on metric spaces that we do not define here for the sake of brevity and we direct the reader to the papers cited within the proof. 

\begin{proof}
All statements in the proof are quantitative. Since $D$ satisfies condition \ref{llc1}, we conclude that $D$ has bounded turning, so the metric $\rho_D$ is comparable to the Euclidean metric. Hence, the space $(D,\rho_D)$ satisfies condition \ref{llc2}. Since $\Omega$ is an inner uniform domain, it is \textit{$2$-Loewner space} when equipped with the metric $\rho_{\Omega}$; this is a consequence of Theorem 6.4 and Remark 6.6 in \cite{BonkHeinonenKoskela:gromov_hyperbolic}. Equivalently, in the terminology of \cite{Heinonen:John}, the domain $\Omega$ is \textit{broad}. Theorem 6.5 in \cite{Heinonen:John} implies that a quasiconformal map $f$ as in the statement from a broad domain $\Omega$ onto a domain $(D,\rho_D)$ satisfying condition \ref{llc2} (with the metric $\rho_D$) is \textit{weakly quasisymmetric} in the metrics $\rho_{\Omega}$ and $\rho_D$. Since $\Omega$ is a broad domain, the space $(\Omega,\rho_{\Omega})$ \textit{doubling} \cite{Vaisala:cylindrical}*{Lemma 2.18}. Also, $D$ is a doubling space as a subset of $\R^2$, so $(D,\rho_D)$ is also doubling. A weakly quasisymmetric map between doubling and connected spaces is quasisymmetric \cite{Heinonen:metric}*{Theorem 10.19}. Therefore, the map $f\colon (\Omega,\rho_\Omega)\to (D,\rho_D)$ is quasisymmetric. By Lemma \ref{lemma:qs_uniform_curves} and Theorem \ref{theorem:vaisala}, $D$ is an inner uniform domain. Finally, we note that an inner uniform domain with bounded turning is uniform \cite{Vaisala:inner}*{Theorem 3.9}.
\end{proof}

\begin{remark}\label{remark:invariance_bounded}
If instead we assume that $\Omega$ and $D$ are bounded domains and $f\colon \Omega \to D$ is quasiconformal, the conclusion of Theorem \ref{theorem:conformal_inner_uniform} remains true, but the distortion function of the quasisymmetry and the uniformity constant of $D$ in the conclusion have to depend, in addition, to a constant $M\geq 1$ such that
$$\frac{\diam\Omega}{\dist(z_0,\partial \Omega)} \leq M \quad \textrm{and}\quad \frac{\diam D}{\dist(f(z_0),\partial D)}\leq M,$$
where $z_0$ is some fixed point in $D$. This can be proved exactly as above, by applying Theorem 6.1 from \cite{Heinonen:John} in place of Theorem 6.5.
\end{remark}

\section{Uniformization of inner uniform domains}\label{section:uniformization}

In the entire section we use the topology of $\widehat{\C}$. Subsets of $\C$ are always understood to be equipped with the Euclidean metric.  

\subsection{Carath\'eodory's kernel convergence}
Let $\Omega_n\subset \widehat{\C}$, $n\in \N$, be a sequence of domains and let $z_0\in \widehat{\C}$ be a point with $z_0\in \Omega_n$ for each $n\in \N$. The \textit{$z_0$-kernel} of $\{\Omega_n\}_{n\in \N}$ is the domain $\Omega$ that is the union of all domains $U$ with the property that $z_0\in U$ and  for each compact set $K\subset U$ there exists $N\in \N$ such that $K\subset \Omega_n$ for all $n\geq N$. Note that the $z_0$-kernel could be the empty set. Moreover, if $\Omega\neq\emptyset$, then $z_0\in \Omega$, $\Omega$ is connected, and for each compact set $K\subset \Omega$ there exists $N\in \N$ such that $K\subset \Omega_n$ for all $n\geq N$. We say that the sequence $\{\Omega_n\}_{n\in \N}$ \textit{converges to a domain $\Omega\subset \widehat{\C}$ in the Carath\'eodory sense with base at $z_0$} if $\Omega$ is the $z_0$-kernel of every subsequence of $\{\Omega_n\}_{n\in \N}$. 

We will use the following version of Carath\'eodory's theorem for multiply connected domains. 

\begin{theorem}[\cite{Goluzin:complex}*{Theorem V.5.1, p.~228}]\label{theorem:caratheodory}
For each $n\in \N$ consider domains $\Omega_n,D_n\subset \widehat{\C}$ such that $\infty\in \Omega_n$ and $\infty\in D_n$, and a conformal map $f_n\colon \Omega_n\to D_n$ with Laurent expansion
$$f_n(z)=z+a_{0,n}+ \frac{a_{1,n}}{z}+\frac{a_{2,n}}{z^2}+\dots$$
in a neighborhood of $\infty$. Suppose that $\{\Omega_n\}_{n\in \N}$ converges to a domain $\Omega$ in the Carath\'eodory sense with base at $\infty$. Then $\{f_n\}_{n\in \N}$ converges locally uniformly in $\Omega$ to a conformal map $f$ if and only if $\{D_n\}_{n\in \N}$ converges to a domain $D$ in the Carath\'eodory sense with base at $\infty$. In that case, $f(\Omega)=D$.
\end{theorem}

The next lemma reveals the relation between Carath\'eodory and Hausdorff convergence. Recall that the Hausdorff distance between two sets $E,F\subset {\C}$, denoted by $d_H(E,F)$, is defined to be the infimum of $r>0$ such that $E$ is contained in the open $r$-neighborhood $N_r(F)$ of $F$ and $F$ is contained in $N_r(E)$ (using the Euclidean metric). The set of all compact subsets of ${\C}$ that are contained in a fixed closed ball is a compact metric space with the Hausdorff distance \cite{BuragoBuragoIvanov:metric}*{Section 7.3.1}.

\begin{lemma}[Carath\'eodory vs Hausdorff convergence]\label{lemma:caratheodory}
Let $\Omega_n\subset \widehat{\C}$, $n\in \N$, be a sequence of domains such that $\infty\in \Omega_n$ for each $n\in \N$. Let $E_n=\widehat{\C}\setminus \Omega_n$, $n\in \N$, and suppose that the sequence $\{E_n\}_{n\in \N}$ converges to a compact set $E\subset \C$ in the Hausdorff sense.  Denote by $\Omega$ the connected component of $\widehat{\C}\setminus E$ that contains $\infty$. Then the $\infty$-kernel of $\{\Omega_n\}_{n\in \N}$ is $\Omega$. In particular, $\{\Omega_n\}_{n\in \N}$ converges to $\Omega$ in the Carath\'eodory sense with base at $\infty$.
\end{lemma}

\begin{proof}
Let $K$ be a compact subset of $\Omega$. There exists $\varepsilon>0$ such that $K\cap N_{\varepsilon}(E)=\emptyset$. Since $E$ is the Hausdorff limit of $\{E_n\}_{n\in \N}$, there exists $N_1\in \N$ such that $E_n\subset N_{\varepsilon}(E)$ for every $n\ge N_1$. Hence, $K\cap E_n=\emptyset$ and  $K\subset \Omega_n$ for every $n\ge N_1$. This implies that the $\infty$-kernel $\{\Omega_n\}_{n\in \N}$ is non-empty and contains $\Omega$. We denote by $U$ the $\infty$-kernel of $\{\Omega_n\}_{n\in \N}$.

Now, we show the reverse inclusion. We claim that $U\cap E=\emptyset$. If not, there exist $z\in E$ and $r>0$ such that $\overline{B}(z,r)\subset U$. By the definition of the $\infty$-kernel, there is $N_2\in \N$ such that $\overline{B}(z,r)\subset \Omega_n$ for every $n\ge N_2$ and hence $ \overline{B}(z,r)\cap E_n=\emptyset$ for every $n\ge N_2$.  This contradicts the assumption that $E$ is the Hausdorff limit of $\{E_{n}\}_{n\in \N}$. Thus, $U\cap E=\emptyset$. Since $U$ is connected, we conclude that it is contained in a component of $\widehat{\C}\setminus E$. Since $\infty\in U$, we have $U\subset \Omega$.
\end{proof}

\subsection{Limits of uniform domains}

We extend the notion of a uniform curve by allowing the endpoints to lie on the boundary of a domain. If $\Omega\subset \C$ is a domain and $A\geq 1$, a curve $\gamma\colon [0,1]\to \br \Omega$ is called an $A$-uniform curve if $\ell(\gamma)\leq A|\gamma(0)-\gamma(1)|$ and $\cig_\ell(\gamma,A)\subset \Omega$. Note that if $\gamma$ is not constant in a neighborhood of $0$ and $1$, then we necessarily have $\gamma((0,1))\subset \Omega$. 

\begin{lemma}[Limits of uniform curves]\label{uniformcurvesconvergence}
Let $\Omega_n\subset {\C}$, $n\in \N$, be a sequence of domains. Let $E_n={\C}\setminus \Omega_n$, $n\in \N$, and suppose that the sequence $\{E_n\}_{n\in \N}$ converges to a compact set $E\subset  \C$ in the Hausdorff sense. For $A\geq 1$, let $\gamma_n\colon [0,1]\to  \br {\Omega_n}$, $n\in \N$, be a sequence of  $A$-uniform curves, parametrized by rescaled arclength, such that $\{\gamma_n(0)\}_{n\in \N}$ and $\{\gamma_n(1)\}_{n\in \N}$ converge to distinct points $z_0$ and $z_1$, respectively. 
Then there exists a component $\Omega$ of ${\C}\setminus E$ and a subsequence of $\{\gamma_n\}_{n\in \N}$ that converges uniformly to an $A$-uniform curve $\gamma\colon [0,1]\to \br \Omega$.
\end{lemma}

\begin{proof}
For $n\in \N$ we have $\ell(\gamma_n)\leq A |\gamma_n(0)-\gamma_n(1)|$ and
\begin{align}\label{assumptionuniform}
\min\{\ell (\gamma_n|_{[0,t]}), \ell(\gamma_n|_{[t,1]} )\} \leq A \dist(\gamma_n(t),\partial \Omega_n),
\end{align}
for every $t\in [0,1]$. The first condition implies that the curves $\gamma_n$, $n\in \N$, are contained in a bounded region. Since the lengths $\{\ell(\gamma_n)\}_{n\in \N}$ are bounded and each $\gamma_n$ is parametrized by rescaled arclength, by the Arzel\`a--Ascoli theorem \cite{BuragoBuragoIvanov:metric}*{Theorem 2.5.14}, there exists a subsequence $\{\gamma_{k_n}\}_{n\in\N}$ of $\{\gamma_n\}_{n\in \N}$ such that $\{\gamma_{k_n}\}_{n\in\N}$ converges uniformly to a curve $\gamma\colon [0,1]\to \C$. Moreover, for every interval $[s,t]\subset[0,1]$ we have
$$\ell(\gamma|_{[s,t]})\leq \liminf_{n\to\infty}\ell( \gamma_{k_n}|_{[s,t]})\leq A|\gamma(0)-\gamma(1)|.$$
In particular, $\ell(\gamma)\leq A |\gamma(0)-\gamma(1)|$. By \eqref{assumptionuniform}, for every $t\in[0,1]$ we have
\begin{align}\label{uniforminequality}
\min \{\ell (\gamma|_{[0,t]}), \ell(\gamma|_{[t,1]} ) \}&\le \min \{\liminf_{n\to \infty} \ell (\gamma_{k_n}|_{[0,t]}), \liminf_{n\to \infty} \ell (\gamma_{k_n}|_{[t,1]})\} \notag \\
&\le \liminf_{n\to \infty} \min \{\ell (\gamma_{k_n}|_{[0,t]}),\ell (\gamma_{k_n}|_{[t,1]})\} \notag \\
&\le A\liminf_{n\to \infty} \dist(\gamma_{k_n}(t),\partial \Omega_{k_n})=A\liminf_{n\to \infty} \dist(\gamma_{k_n}(t),E_{k_n}) \notag\\
&=A\dist(\gamma(t),E),
\end{align}
where the last equality follows from the facts that the sequence $\{E_n\}_{n\in \N}$ converges to $E$ in the Hausdorff sense and $\{\gamma_{k_n}(t)\}_{n\in\N}$ converges to $\gamma(t)$. We have $\gamma(0)=z_0$ and $\gamma(1)=z_1$. By assumption, $z_1\neq z_0$, so $\gamma$ is a non-constant curve. Let $[t_0,t_1]\subset [0,1]$ be such that $\gamma([0,t_0])=\{z_0\}$, $\gamma([t_1,1])=\{z_1\}$, and $\gamma|_{[t_0,t_1]}$ is not constant in a neighborhood of $t_0$ or $t_1$. By \eqref{uniforminequality} we infer that $\gamma(t)\in {\C}\setminus E$ for every $t\in(t_0,t_1)$. Since $\gamma((t_0,t_1))$ is connected, there exists a component $\Omega$ of ${\C}\setminus E$ such that $\gamma( [0,1])\subset \br \Omega$ and hence $\dist(\gamma(t),E)=\dist(\gamma(t),\partial \Omega)$ for $t\in [0,1]$. In combination with \eqref{uniforminequality}, we obtain that $\gamma\colon [0,1]\to \br \Omega$ is an $A$-uniform curve. 
\end{proof}

\begin{remark}\label{remark:uniform_boundary}
It is a consequence of Lemma \ref{uniformcurvesconvergence} that if $\Omega$ is an $A$-uniform domain then any two points in $\br\Omega$ can be connected by an $A$-uniform curve. 
\end{remark}

\begin{lemma}[Limits of uniform domains]\label{lemma:uniform:convergence}
Let $A\geq 1$ and $\Omega_n\subset {\C}$, $n\in \N$, be a sequence of $A$-uniform domains. Let $E_n={\C}\setminus \Omega_n$, $n\in \N$, and suppose that the sequence $\{E_n\}_{n\in \N}$ converges to a compact set $E\subset  \C$ in the Hausdorff sense.  Then $\Omega={\C}\setminus E$ is an $A$-uniform domain. Moreover, for each component $F$ of $E$ and for each $n\in \N$ there exists a component  $F_n$ of $E_n$ such that the sequence $\{F_n\}_{n\in \N}$ converges in the Hausdorff sense to $F$.
\end{lemma}

The proof relies on a result of Martio--Sarvas \cite{MartioSarvas:uniform}*{Theorem 2.24} (see also Gehring \cite{Gehring:schwarzian}*{Lemma 5}) on the geometry of uniform domains. It implies that each bounded boundary component of a uniform domain in $\C$ is either a point or a Jordan curve (in fact, a quasicircle).  

\begin{proof}
Let $x,y\in {\C}\setminus E$ be distinct points. By Hausdorff convergence, there exists $N_1\in \N$ such that $x,y\notin E_n$ and $x,y\in \Omega_n$ for $n\geq N_1$. By assumption, for each $n\geq N_1$, the domain $\Omega_n$ is $A$-uniform and thus there is an $A$-uniform curve $\gamma_n\colon [0,1]\to {\Omega_n}$ with $\gamma_n(0)=x$ and $\gamma_n(1)=y$. We parametrize $\gamma_n$ by rescaled arclength. By Lemma \ref{uniformcurvesconvergence} there exists a component $U$ of ${\C}\setminus E$ such that a subsequence of $\{\gamma_n\}_{n\geq N_1}$ converges to an $A$-uniform curve $\gamma\colon [0,1]\to \br U$ with $\gamma(0)=x$ and $\gamma(1)=y$. In particular, for $t\in [0,1]$ we have
$$\min \{\ell (\gamma|_{[0,t]}), \ell(\gamma|_{[t,1]} ) \}\le A\dist(\gamma(t),\partial U).$$
Since $x,y\notin E$, we conclude that $\dist(\gamma(t),\partial U)>0$ for all $t\in [0,1]$. Thus, $\gamma([0,1])\subset {\C}\setminus E$. Since the points $x,y\in {\C}\setminus E$ were arbitrary, we conclude that the set $\Omega={\C}\setminus E$ is connected and any pair of points of $\Omega$ can be joined by an $A$-uniform curve in $\Omega$. Thus, $\Omega$ is an $A$-uniform domain. 

Let $F$ be a component of $E$. Since $\Omega={\C}\setminus E$ is a uniform domain, by \cite{MartioSarvas:uniform}*{Theorem 2.24} the set $F$ is a single point or a closed Jordan region. Let $x\in F$ and suppose that $x\in \inter F$ if $F$ is a closed Jordan region.  By assumption, the sequence $\{E_n\}_{n\in \N}$ converges to $E$ in the Hausdorff sense. Thus, there exist points $x_n\in E_n$, $n\in \N$, such that $x_n\to x$ as $n\to \infty$. For each $n\in \N$, let $F_n$ be the component of $E_n$ containing $x_n$. Consider a subsequence $\{F_{k_n}\}_{n\in \N}$ of $\{F_n\}_{n\in \N}$ that converges in the Hausdorff sense to a compact and connected set $V\subset E$. Then $x\in V\cap F$. Since $F$ is a component of $E$, we conclude that $V\subset F$. We claim that $V=F$. Assuming that this is the case, then the original sequence $\{F_n\}_{n\in \N}$ must also converge to $F$, as desired. 

We now prove the claim. If $F$ is a single point there is nothing to prove, so assume that $F$ is a closed Jordan region and $x\in \inter F$. For the sake of contradiction, we assume that $ F\setminus V\neq \emptyset$ and let $z\in F\setminus V$. Since $\partial F$ is a Jordan curve, there exists a curve $\alpha \colon [0,1]\to F$ such that $\alpha(0)=z$ and $\alpha(1)=x$ and $\alpha ((0,1])\subset \inter F$. So, there is a point $y\in \partial V\cap \alpha ((0,1])$ and hence $y\in \partial V\cap \inter F$. Next, we prove that 
\begin{align}\label{lemma:uniform:convergence:claim2}
\textrm{there exists a sequence $y_n\in \partial F_{k_n}$, $n\in \N$, that converges to $y$ as $n\to\infty$.}
\end{align}
Let $\varepsilon>0$. Then $(\C\setminus V)\cap B(y,\varepsilon)\neq \emptyset$. Since $\{F_{k_n}\}_{n\in \N}$ converges in the Hausdorff sense to $V$, there is $N_2\in \N$ such that $F_{k_n}\cap B(y,\varepsilon)\neq \emptyset$ and $({\C}\setminus F_{k_n})\cap B(y,\varepsilon)\neq \emptyset$ for every $n\ge N_2$. By the connectedness of $B(y,\varepsilon)$, we have $\partial F_{k_n}\cap B(y,\varepsilon)\neq \emptyset$ for $n\geq N_2$. By varying $\varepsilon$, one can obtain a sequence $\{y_n\}_{n\in \N}$ as in \eqref{lemma:uniform:convergence:claim2}.

Let $x_0\in \Omega={\C}\setminus E$, so there exists $N_3\in \N$ such that $x_0\in \Omega_{k_n}$ for all $n\geq N_3$. For each $n\geq N_3$, there is an $A$-uniform curve $\gamma_n\colon [0,1]\to  \br {\Omega_{k_n}}$ joining $y_n$ to $x_0$; see Remark \ref{remark:uniform_boundary}. By Lemma \ref{uniformcurvesconvergence}, a subsequence of $\{\gamma_n\}_{n\in \N}$ converges to an $A$-uniform curve $\gamma\colon [0,1]\to \br \Omega$ joining $y$ to $x_0$. However $y\in \inter F$ and thus $\gamma([0,1])$ is not contained in $\br \Omega$, which is a contradiction. So, $V=F$ and the proof is complete.
\end{proof}

\begin{corollary}[Limits of uniform circle domains]\label{corollary:circle}
Let $A\geq 1$ and $D_n\subset \widehat{\C}$, $n\in \N$, be a sequence of domains such that $\infty\in D_n$ and $D_n\cap \C$ is an $A$-uniform circle domain for each $n\in \N$.  Suppose that $\{D_n\}_{n\in \N}$ converges to a domain $D$ in the Carath\'eodory sense with base at $\infty$. Then $D\cap \C$ is an $A$-uniform circle domain.
\end{corollary}
\begin{proof}
By the convergence of $\{D_n\}_{n\in \N}$ to $D$ with base at $\infty$, there exists a neighborhood of $\infty$ that is contained in $D_n$ for all $n\in \N$. Therefore the sets $E_n=\widehat{\C}\setminus D_n$, $n\in \N$, are contained in a ball $B(0,R)$ for all $n\in \N$. By passing to a subsequence we assume that $\{E_n\}_{n\in \N}$ converges to a compact set $E\subset {\C}$ in the Hausdorff sense.  By Lemma \ref{lemma:uniform:convergence}, ${\C}\setminus E$ is an $A$-uniform domain. Moreover, each component of $E$ arises as a Hausdorff limit of a sequence of components $F_n$ of $E_n$, $n\in \N$. Therefore, ${\C}\setminus E$ is a circle domain. Finally, by Lemma \ref{lemma:caratheodory} we have $D=\widehat{\C}\setminus E$.
\end{proof}

\subsection{Uniformization by circle domains}\label{section:uniformization_euclidean}

\begin{proof}[Proof of Theorem \ref{theorem:uniformization_euclidean}]
If $U\subset \C$ is a domain containing a neighborhood of $\infty$, we denote by $\widehat{U}$ the domain $U\cup \{\infty\}$ in $\widehat{\C}$.

Suppose that $\Omega$ is an inner $A$-uniform domain for some $A\geq 1$. By Corollary \ref{corollary:approximation} there exists a sequence of finitely connected inner $C_1(A)$-uniform domains $\Omega_n$, $n\in \N$, such that $\Omega\subset \Omega_{n+1}\subset \Omega_{n}$ for all $n\in \N$, and the set $(\bigcap_{n=1}^\infty \Omega_n) \setminus \Omega$ is totally disconnected. In particular, if $V$ is the $\infty$-kernel of $\{\widehat \Omega_n\}_{n\in \N}$, then $\bigcap_{n=1}^\infty \widehat \Omega_n\supset V\supset \widehat \Omega$ and $V\setminus \widehat\Omega$ is a totally disconnected closed subset of $V$.  Moreover, since the sequence of domains $\{\widehat \Omega_n\}_{n\in \N}$ is decreasing, $V$ is also the $\infty$-kernel of every subsequence of that sequence.

By Koebe's theorem \cite{Koebe:FiniteUniformization} (see also \cite{Conway:complex2}*{Theorem 15.7.9}), for each $n\in \N$ there exists a (unique) conformal map $f_n$ from $\Omega_n$ onto a finitely connected circle domain $D_n\subset {\C}$ such that $f_n$ satisfies the normalization
\begin{align}\label{lemma:normalization}
f_n(z)= z + \frac{a_{1,n}}{z}+\frac{a_{2,n}}{z^2}+\dots\quad \textrm{in a neighborhood of $\infty$}.
\end{align}
By Theorem \ref{theorem:conformal_inner_uniform} and Lemma \ref{lemma:circle_llc}, $D_n$ is a $C_2(A)$-uniform circle domain for each $n\in \N$. Note that $f_n$ extends to $\infty$ and $f_n(\infty)=\infty$, $n\in \N$. The family of conformal maps $\{f_n|_V\}_{n\in \N}$ that satisfy the normalization in \eqref{lemma:normalization} is normal and compact; see \cite{Schiff:normal_families}*{Section 5.1, pp.~165--167}.  After passing to a subsequence,  we may assume that $\{f_n|_V\}_{n\in \N}$ converges locally uniformly in $V$ to a conformal map $f$ on $V$. By Theorem \ref{theorem:caratheodory}, the domains $\widehat D_n=f(\widehat \Omega_n)$, $n\in \N$, converge in the Carath\'eodory sense to a domain $\widehat D$ with base at $\infty$ such that $f(V)=\widehat D$. By Corollary \ref{corollary:circle}, the domain $D=\widehat D\cap \C$ is a $C_2(A)$-uniform circle domain.

We claim that $f(\Omega)$ is a circle domain. Let $E=f(V\setminus \widehat \Omega)$ and note that $E$ is a totally disconnected closed subset of $\widehat D$, because $f$ is a homeomorphism.  Note that $f(\widehat \Omega)=\widehat D\setminus E$. In order to show that $f( \Omega)$ is a circle domain, it suffices to show that the components of $\widehat{\C}\setminus f(\widehat \Omega)$ are precisely the components of $\widehat{\C}\setminus \widehat D$ and the components of $E$. The argument is purely topological. Let $F$ be a component of $\widehat{\C}\setminus f(\widehat \Omega)= (\widehat{\C}\setminus \widehat D)\cup E$. Suppose that $F\cap E\neq \emptyset$ and let $x\in F\cap E$. By Zoretti's theorem (Theorem \ref{theorem:zoretti}), since $x$ is a component of $E\subset \widehat{D}$, there exists a Jordan region $W\subset \br W\subset  \widehat D$ such that $x\in W$ and $\partial W\cap E=\emptyset$. In particular, $\partial W\subset f(\widehat \Omega)$. By the connectedness of $F$ we must have $F\subset W\subset \widehat D$. We conclude that $F\cap (\widehat{\C}\setminus \widehat D)=\emptyset$, so $F\subset E$. Since $E$ is totally disconnected, we have $F=\{x\}$. Next, suppose that $F\cap E=\emptyset$, so $F\subset \widehat{\C}\setminus \widehat D$.  Then $F$ is contained in a component of $\widehat{\C}\setminus \widehat D$, but each such component is contained in a component of $\widehat{\C}\setminus f(\widehat \Omega)$. Thus, $F$ must coincide with a component of $\widehat{\C}\setminus \widehat D$, as desired.

Now, we discuss the uniqueness of $f$. Suppose that there exists another conformal map $g$ from $\Omega$ onto a circle domain that contains a neighborhood of $\infty$ so that $g$ extends continuously to a map that fixes $\infty$. Then $g\circ f^{-1}$ is a conformal map from the circle domain $f(\Omega)$ onto another circle domain in $\C$ and $g\circ f^{-1}$ extends to a map that fixes $\infty$. Since $f(\Omega)$ is a uniform domain, it is also a {John domain} (i.e., any two points of the domain can be joined by a curve $\gamma$ so that $\cig_{\ell}(\gamma,A')$ is contained in the domain for some uniform constant $A'\geq 1$). By \cite{NtalampekosYounsi:rigidity}*{Corollary 1.7}, John circle domains are {conformally rigid} (as defined in the Introduction). In particular, $g\circ f^{-1}$ is the restriction of a conformal automorphism of $\C$.
\end{proof}

\begin{proof}[Proof of Theorem \ref{theorem:qs_characterization}]
All statements in the proof are quantitative. Suppose that $f$ is a quasisymmetry from $(\Omega,\rho_{\Omega})$ onto a uniform domain $D\subset \C$. Since uniform domains have bounded turning, $\rho_D$ is comparable to the Euclidean metric. Thus, $f\colon (\Omega,\rho_\Omega)\to (D,\rho_D)$ is quasisymmetric. This implies that $f^{-1}\colon (D,\rho_D)\to (\Omega,\rho_{\Omega})$ is also quasisymmetric \cite{Heinonen:metric}*{Proposition 10.6}.  Since $D$ is uniform, it is also inner uniform and by Theorem \ref{theorem:vaisala} there exists $A\geq 1$ such that any two points can be connected by an $(A,\rho_D)$-uniform curve.  By Lemma \ref{lemma:qs_uniform_curves}, any two points of $\Omega$ can be connected by a $(C(A),\rho_{\Omega})$-uniform curve. Theorem \ref{theorem:vaisala} implies that $\Omega$ is an inner uniform domain, as desired. 

Conversely, suppose that $\Omega$ is an inner uniform domain that  contains a neighborhood of $\infty$. By Theorem \ref{theorem:uniformization_euclidean}, there exists a conformal map $f$ from $\Omega$ onto a uniform circle domain $D$ that contains a neighborhood of $\infty$, and $f$ extends to $\infty$ so that $f(\infty)=\infty$. By Lemma \ref{lemma:circle_llc}, $D$ is $1$-LLC. Theorem \ref{theorem:conformal_inner_uniform} implies that $f\colon (\Omega,\rho_\Omega)\to D$ is quasisymmetric. 
\end{proof}

We also include a version of Theorem \ref{theorem:uniformization_euclidean} for bounded domains. 

\begin{theorem}\label{theorem:bounded_uniformization}
Let $\Omega\subset \C$ be a bounded inner uniform domain. Then there exists a conformal map $f$ from $\Omega$ onto a bounded circle domain $D$. Moreover, the conformal map $f$ is unique up to postcomposition with M\"obius transformations and the circle domain $D$ is a uniform domain, quantitatively. 
\end{theorem}
\begin{proof}
The same proof scheme as above applies with appropriate adaptations to bounded domains, so we only provide a sketch. Let $z_0\in \Omega$ be a point to be determined. By Corollary \ref{corollary:approximation} the domain $\Omega$ can be approximated by finitely connected domains $\Omega_n\supset \Omega$, $n\in \N$, such that the unbounded component $S_0$ of $\R^2\setminus \Omega$ is also a component of $\R^2\setminus \Omega_n$, so $\diam \Omega=\diam \Omega_n$, $n\in \N$.  For each $n\in \N$ there exists a conformal map $f_n$ from $\Omega_n$ onto a circle domain $D_n$ such that $D_n\subset \D$, $\partial \D$ is a component of $\partial D_n$ corresponding to $\partial S_0$, $f_n(z_0)=0$, and $f_n'(z_0)>0$. Our task is to show that the domains $D_n$ are uniform domains, quantitatively. Assuming that, one can follow the rest of the proof of Theorem \ref{theorem:uniformization_euclidean} with appropriate adaptations.

We fix $n\in \N$. In the previous proof the fact that $D_n$ is uniform was provided by Theorem \ref{theorem:conformal_inner_uniform}.  Since the domains $\Omega_n,D_n$ are bounded, by Remark \ref{remark:invariance_bounded} the uniformity constant of $D_n$ depends on the inner uniformly constant of $\Omega$ and on a constant $M\geq 1$ such that
\begin{align}\label{theorem:bounded_uniformization_m}
\frac{\diam\Omega_n}{\dist(z_0,\partial \Omega_n)}\leq \frac{\diam\Omega}{\dist(z_0,\partial \Omega)}  \leq M \quad \textrm{and}\quad \frac{2}{\dist(0,\partial D_n)}\leq M.
\end{align}
Suppose that $\Omega$ is inner $A$-uniform for some $A\geq 1$. Let $x_1,x_2\in \Omega$ such that $\diam \Omega\leq 2|x_1-x_2|$. Consider an inner $A$-uniform curve $\gamma$ connecting $x_1,x_2$. Condition \eqref{uniform:cigar} (using the Euclidean metric) implies that there exists a point $z_0$ on $\gamma$ such that 
$$\dist(z_0,\partial \Omega)\geq \frac{1}{2A} \ell(\gamma) \geq \frac{1}{4A}\diam \Omega.$$
This implies that the first inequality in \eqref{theorem:bounded_uniformization_m} holds with $M\geq 4A$. Let $r_0'=\dist(0,\partial D_n)$. We will show that $r_0'$ is bounded from below, depending only on $A$. This will complete the proof that the domain $D_n$ is uniform, quantitatively. 

Let $r_0=\dist(z_0,\partial \Omega)$. By Zoretti's theorem (Theorem \ref{theorem:zoretti}) there exists a Jordan curve $J_n\subset \Omega\subset \Omega_n$ that separates $S_0$ from $E=\br B(z_0,r_0/2)$ and $f_n(J_n)\subset N_{r_0'}(\partial \D)$. In particular, $\diam J_n \geq \diam E$. Also, $\dist_{\rho_{\Omega_n}}(J_n, E) \leq \diam \Omega_n=\diam \Omega \leq 4Ar_0$. Therefore,
\begin{align*}
\frac{\dist_{\rho_{\Omega_n}}(J_n,E)}{\min\{\diam J_n,\diam E\}}  \leq 4A.
\end{align*}
Since $\Omega_n$ is inner uniform, the space $(\Omega_n,\rho_{\Omega_n})$ is $2$-Loewner, quantitatively; see  \cite{BonkHeinonenKoskela:gromov_hyperbolic}*{Theorem 6.4} and the preceding discussion. Thus, the $2$-modulus $\mod \Gamma(J_n,E;\Omega_n)$ of the family of curves connecting $J_n$ and $E$ in $\Omega_n$ is uniformly bounded from below; see \cite{Heinonen:metric}*{Section 7.3} for the definition of modulus. By the conformal invariance of modulus,
$$\mod \Gamma(f_n(J_n),f_n(E); D_n) \geq C(A).$$ 
By Koebe's distortion theorem \cite{Pommerenke:conformal}*{Theorem 1.3}, $f_n(E)\subset \br B(0,8r_0')$. Without loss of generality, $r_0'<1/9$, so the ball $\br B(0,8r_0')$ is disjoint from $\partial B(0,1-r_0')$. Any curve connecting $f_n(J_n)$ and $f_n(E)$ in $D_n$ has a subcurve in the family $\Gamma$ of curves in the plane connecting $\partial B(0,8r_0')$ and $\partial B(0,1-r_0')$. By \cite{Heinonen:metric}*{Example 7.14}, 
$$2\pi \left(\log \frac{1-r_0'}{8r_0'}\right)^{-1}=\mod \Gamma \geq \mod \Gamma(f_n(J_n),f_n(E); D_n) \geq C(A).$$
Thus, $r_0'$ is uniformly bounded from below, as desired. 
\end{proof}

\begin{remark}\label{remark:qs_bounded}
Theorem \ref{theorem:qs_characterization} also extends to bounded domains as follows. If $\Omega$ is a bounded inner uniform domain, by Theorem \ref{theorem:bounded_uniformization} there exists a conformal map $f$ from $\Omega$ onto a bounded uniform circle domain $D$.  The fact that $f\colon (\Omega,\rho_\Omega)\to D$ is quasisymmetric follows from Remark \ref{remark:invariance_bounded}. The dependence of the quasisymmetric distortion function solely on the inner uniformity constant of $\Omega$ can be shown as in the proof of Theorem \ref{theorem:bounded_uniformization}.
\end{remark}

\section{Uniformization of Gromov hyperbolic domains}\label{section:uniformization_gh}
\subsection{Conformal invariance of hyperbolicity}

Recall that the spherical metric $\sigma$ on $\widehat \C$ is given by
\begin{align*}
\sigma(z,w)= \inf_{\gamma}\int_{\gamma} \frac{2\,|dz|}{1+|z|^2}
\end{align*}
where $z,w\in \widehat{\C}$ and the infimum is taken over all rectifiable curves $\gamma$ in $\widehat{\C}$ that connect $z$ and $w$. A convenient way to estimate the spherical metric is through the chordal metric $\chi$, defined by 
$$\chi(z,w)= \frac{2|z-w|}{\sqrt{1+|z|^2} \sqrt{1+|w|^2}} \quad \textrm{and}\quad \chi({z,\infty})=\frac{2}{\sqrt{1+|z|^2}},$$
where $z,w\neq\infty$. We have $\chi\leq \sigma \leq \frac{\pi}{2} \chi$ and if $e$ denotes the Euclidean metric, then $\chi\leq \sigma\leq 2e$. In what follows we use subscripts to denote the metric that is used in the various metric notions. For example, a Euclidean open ball is denoted by $B_e(z,r)$ and the chordal diameter of a set $E$ is denoted by $\diam_{\chi}E$. 

A map $f\colon (X,d_X)\to (Y,d_Y)$ between metric spaces is \textit{bi-Lipschitz} if there exists a constant $L\geq 1$ such that 
$$L^{-1}d_X(x,y)\leq d_Y(f(x),f(y))\leq Ld_X(x,y)$$
for every $x,y\in X$. In this case we say that $f$ is $L$-bi-Lipschitz. Gromov hyperbolicity is conformally invariant. Specifically, by \cite{GehringOsgood:uniform}*{Theorem 3}, conformal maps are {bi-Lipschitz} in the \textit{Euclidean} quasihyperbolic metric. In turn, bi-Lipschitz maps preserve hyperbolicity, quantitatively \cite{GhysHarpe:gromov}*{Theorem 5.2.12, p.~88}. Here, the Euclidean quasihyperbolic metric in a domain $\Omega\subsetneq \C$ is defined by using the Euclidean distance to the boundary of a domain rather than the spherical one, and integrating against the Euclidean length element. 

In order to distinguish between the various quasihyperbolic metrics, we use superscripts: $k^e_{\Omega}$, $k^{\sigma}_{\Omega}$ denote the Euclidean and spherical quasihyperbolic metrics in a domain $\Omega\subset \C$, respectively. 

Since we have defined hyperbolicity using the spherical quasihyperbolic metric, the above results do not imply that conformal images of hyperbolic spaces are hyperbolic in a quantitative fashion. The next result provides some quantitative dependence of the constants that is sufficient for our purposes.

\begin{theorem}\label{theorem:conformal_bilip}
Let $a>0$ and  $U,V\subset \widehat{\C}$ be domains with $\diam_{\sigma}(\widehat{\C}\setminus U)\geq a$ and $\diam_{\sigma} (\widehat{\C}\setminus V)\geq a$. Then every conformal map $f\colon U\to V$ is $L(a)$-bi-Lipschitz in the spherical quasihyperbolic metrics. In particular, if $U$ is $\delta$-hyperbolic for some $\delta>0$, then $V$ is $\delta'(a,\delta)$-hyperbolic.
\end{theorem}

Note that the dependence on $a$ in the first part of the theorem is necessary. Indeed, consider the map $f(z)=rz$ on $\D$, where $r>0$. Then on the domain $B_r=B_e(0,r)$, upon estimating the spherical metric by the chordal metric, we have
$$k_{B_r}^{\sigma}(0,r/2) = \int_{0}^{r/2} \frac{2\, dt}{\sigma(t,r)(1+t^2)}  \geq \frac{2}{\pi} \int_0^{r/2} \frac{\sqrt{1+r^2} \, dt}{(r-t) \sqrt{1+t^2}}\geq \frac{2}{\pi} \mathrm{arcsinh}(r/2).$$
Hence, $k_{B_r}^{\sigma}(0,r/2)$ cannot be comparable to $k_{\D}^{\sigma}(0,1/2)$ with uniform constants. 

The proof of the theorem relies on an inequality between the Euclidean and spherical quasihyperbolic metrics. 

\begin{lemma}\label{lemma:comparable}
Let $\Omega\subset \widehat{\C}$ be a domain such that $\infty\notin \Omega$ and $\partial \Omega$ contains at least two points. Then
$$(\pi \sqrt{2})^{-1} k_{\Omega}^e\leq k_{\Omega}^{\sigma}\leq 3(2+D) k_{\Omega}^{e}, \quad \textrm{where} \quad D=\dist_e(0,\widehat{\C}\setminus \Omega).$$
\end{lemma}

The inequality is stated with different constants in \cite{HerronLindquist:quasihyperbolic}*{Fact 2.9} and is obtained from some general estimates on metric space inversions. Here we give a self-contained elementary proof with improved constants.

\begin{proof}
Let $z\in \Omega$ and $w\in \partial \Omega$ such that $\chi(z,w)=\dist_{\chi}(z,\partial \Omega)$. Since $\sigma \geq \chi$,
\begin{align}\label{lemma:qh:upper}
\rho(z)\coloneqq (1+|z|^2)\dist_{\sigma}(z,\partial \Omega)\geq
\begin{cases} \frac{2|z-w|\sqrt{1+|z|^2}}{\sqrt{1+|w|^2}}, & w\neq \infty \vspace{0.3em} \\ 2\sqrt{1+|z|^2}, & w=\infty\end{cases}.
\end{align}
If $\infty>|w|\geq 2|z|$, then the monotonicity of the function $x\mapsto \frac{x-|z|}{\sqrt{1+x^2}}$ implies that
\begin{align}\label{lemma:qh:upper:w>2z}
\frac{2|z-w|\sqrt{1+|z|^2}}{\sqrt{1+|w|^2}} \geq \frac{2(|w|-|z|)\sqrt{1+|z|^2}}{\sqrt{1+|w|^2}} \geq \frac{2|z|\sqrt{1+|z|^2}}{\sqrt{1+4|z|^2}}\geq |z|.
\end{align}
If $|w|<2|z|$, then
\begin{align}\label{lemma:qh:upper:w<2z}
\frac{2|z-w|\sqrt{1+|z|^2}}{\sqrt{1+|w|^2}}\geq \frac{2|z-w|\sqrt{1+|z|^2}}{\sqrt{1+4|z|^2}} \geq |z-w|\geq \dist_e(z,\partial \Omega).
\end{align}

We set $D=\dist_e(0,\partial \Omega)$. If $|z|\geq D/2$, then $|z|\geq 3^{-1}\dist_e(z,\partial \Omega)$.
If $w\neq \infty$, then \eqref{lemma:qh:upper}, \eqref{lemma:qh:upper:w>2z}, and \eqref{lemma:qh:upper:w<2z} give $\rho(z)\geq 3^{-1}\dist_e(z,\partial \Omega)$. If $w=\infty$, then \eqref{lemma:qh:upper} gives $\rho(z)\geq 2|z|\geq \frac{2}{3} \dist_e(z,\partial \Omega) \geq  3^{-1}\dist_e(z,\partial \Omega)$. 

If $|z|<D/2$, then $\dist_e(z,\partial \Omega)\leq 3D/2$. Since $w\in \partial \Omega$, we have $|w|\geq D>2|z|$ in the case that $w\neq \infty$. Thus,
\begin{align*}
\frac{2|z-w|\sqrt{1+|z|^2}}{\sqrt{1+|w|^2}} &\geq \frac{2(|w|-|z|)}{\sqrt{1+|w|^2}} \geq \frac{2 (D-|z|)}{\sqrt{1+D^2}}\geq \frac{ D}{\sqrt{1+D^2}}\geq \frac{2}{3}\frac{\dist_e(z,\partial \Omega)}{\sqrt{1+D^2}}.
\end{align*}
We now have $\rho(z)\geq \frac{2}{3} (1+D^2)^{-1/2}  \dist_e(z,\partial \Omega) \geq \frac{2}{3}(1+D)^{-1}\dist_e(z,\partial \Omega)$. If $w=\infty$, then by \eqref{lemma:qh:upper} we have $\rho(z) \geq 2 \geq \frac{4}{3}D^{-1} \dist_e(z,\partial \Omega)\geq \frac{2}{3}(1+D)^{-1}\dist_e(z,\partial \Omega)$. Combining all above cases, we obtain 
$$2^{-1}\rho(z) \geq  \min\{ 3^{-1} (1+D)^{-1}, 6^{-1}\}\dist_e(z,\partial \Omega)\geq 3^{-1}(2+D)^{-1} \dist_e(z,\partial \Omega).$$
Therefore, $k_{\Omega}^{\sigma}\leq 3(2+D)k_{\Omega}^e$. We have completed the proof of the upper bound.

For the lower bound, let $w\in \partial \Omega\cap \C$ such that $|z-w|=\dist_e(z,\partial \Omega)$. Then 
\begin{align}\label{lemma:comparable:lower}
\rho(z) \leq \pi \frac{|z-w|\sqrt{1+|z|^2}}{\sqrt{1+|w|^2}} \quad \textrm{and}\quad \rho(z) \leq (1+|z|^2) \sigma(z,\infty)\leq \pi\sqrt{1+|z|^2}.
\end{align}
If $|w|\geq |z|/2$, this gives $\rho(z)\leq 2\pi|z-w|=2\pi \dist_e(z,\partial \Omega)$. Suppose that $|w|<|z|/2$ and $|z|<1$. Then we obtain $\rho(z) \leq  \pi\sqrt{2}|z-w| =\pi \sqrt{2} \dist_e(z,\partial \Omega)$. Finally, suppose that $|w|<|z|/2$ and $|z|\geq 1$. By the second inequality in \eqref{lemma:comparable:lower}, we have
\begin{align*}
\rho(z) \leq \pi \sqrt{2} |z|\leq 2\pi\sqrt{2}|z-w|= 2\pi \sqrt{2} \dist_e(z,\partial \Omega).
\end{align*}
In all cases, we obtain $(\pi \sqrt{2})^{-1}k_{\Omega}^e\leq k_{\Omega}^\sigma$.
\end{proof}

\begin{proof}[Proof of Theorem \ref{theorem:conformal_bilip}]
Let $f\colon U\to V$ be a conformal map. We consider isometries $S,T$ of $\widehat{\C}$ such that $\infty \in \widehat \C\setminus  S(U)$ and $\infty \in \widehat \C\setminus  T(V)$. Define $g= T\circ f \circ S^{-1} \colon S(U)\to T(V)$, which is a conformal map between domains in $\C$. By \cite{GehringOsgood:uniform}*{Theorem 3}, $g$ is bi-Lipschitz in the Euclidean quasihyperbolic metric. That is, there exists a uniform constant $L\geq 1$ such that
\begin{align}\label{theorem:conformal_invariance:qi}
L^{-1}k^e_{S(U)}(z,w) \leq k^e_{T(V)}(g(z),g(w))\leq Lk^e_{S(U)}(z,w)
\end{align}
for all $z,w\in S(U)$. By assumption, $\diam_{\sigma}(\widehat\C\setminus T(V))\geq a$, so the set $\widehat \C\setminus  T(V)$ is not contained in $B_{\sigma}(\infty, a/2)$. For each point $z\in \widehat{\C}\setminus T(V)$ with $\sigma(z,\infty)\geq a/2$, we have $\sigma(0,z)\leq \pi-a/2$. Thus,
$$\pi-a/2\geq \sigma(0,z)=\int_0^{|z|} \frac{2dt}{1+t^2}=2\arctan |z|,$$
which implies that  
$$\dist_e(0, \widehat\C\setminus T(V)) \leq |z|\leq  \tan (\pi/2-a/4).$$
The same estimates apply to $S(U)$. Lemma \ref{lemma:comparable} and  \eqref{theorem:conformal_invariance:qi} yield that $g$ is $L(a)$-bi-Lipschitz in the spherical quasihyperbolic metric. Since $S$ and $T$ are spherical isometries, the map $f$ has the same property. The preservation of hyperbolicity as in the last statement of the theorem follows from the fact that $f\colon (U,k_U^{\sigma})\to (V,k_V^{\sigma})$ is bi-Lipschitz, combined with \cite{GhysHarpe:gromov}*{Theorem 5.2.12, p.~88}.
\end{proof}

\begin{lemma}\label{lemma:normal}
	For each $a\geq 1$ there exists $b\geq 1$ such that the following statement is true. Let $U\subset \widehat{\C}$ be a domain with $\infty\in U$, $\widehat \C\setminus U\subset \br B_e(0,a)$, and $\diam_{e}(\widehat \C\setminus U)\geq a^{-1}$. If $f\colon U\to \widehat{\C}$ is a conformal embedding with $f(\infty)=\infty$ and
\begin{align}\label{lemma:normal:normalization}
f(z)= z + \frac{a_1}{z}+\frac{a_2}{z^2}+\dots\quad \textrm{in a neighborhood of $\infty$},
\end{align}
then $\widehat \C\setminus f(U)\subset \br B_e(0,b)$ and $\diam_{e}(\widehat {\C}\setminus f(U))\geq b^{-1}$. 
\end{lemma}
\begin{proof}
Let $V=f(U)$. By assumption, $f$ is univalent in the region $\{z\in \C: |z|>a\}$. By \cite{Schiff:normal_families}*{Lemma 5.1.3} we have $\widehat \C\setminus V\subset \br B_e(0,2a)$. We will show that $\diam_e(\widehat\C\setminus V)$ is bounded from below, depending on $a$. Suppose, for the sake of contradiction that for every $n\in \N$ there exists a domain $U_n$ with $\infty\in U_n$, $\widehat\C\setminus U_n\subset \br B_e(0,a)$, and $\diam_e(\widehat{\C}\setminus U_n)\geq a^{-1}$, and a conformal map $f_n\colon U_n\to V_n$ with $f(\infty)=\infty$ such that $f_n$ satisfies the normalization in \eqref{lemma:normal:normalization} and $\diam_{e}(\widehat{\C}\setminus V_n) \to 0$ as $n\to\infty$. By the above, we have $\widehat\C\setminus V_n\subset \br B_e(0,2a)$ for each $n\in \N$. 

After passing to a subsequence, we may assume that the sequence $\{\widehat\C\setminus V_n\}_{n\in \N}$ converges to a point $w_0\in \br B_e(0,2a)$ in the Hausdorff sense. By Lemma \ref{lemma:caratheodory}, $\{V_n\}_{n\in \N}$ converges to the domain $V=\widehat{\C}\setminus \{w_0\}$ in the Carath\'eodory sense with base at $\infty$. After passing to a further subsequence, we see that $\{\widehat\C\setminus U_n\}_{n\in \N}$ converges in the Hausdorff sense to a compact set $E\subset \br B_e(0,a)$ with $\diam_e E\geq a^{-1}$. By Lemma \ref{lemma:caratheodory}, $\{U_n\}_{n\in \N}$ converges in the Carath\'eodory sense with base at $\infty$ to the domain $U$ that is the component of $\widehat{\C}\setminus E$ that contains $\infty$. In particular, $\diam_e(\widehat\C\setminus U)\geq \diam_e E \geq a^{-1}$. By Theorem \ref{theorem:caratheodory}, $\{f_n\}_{n\in \N}$ converges locally uniformly in $U$ to a conformal map $f\colon U\to V$. This is a contradiction, since $\partial U$ has at least two points and $\partial V$ is a single point.
\end{proof}

\begin{corollary}\label{corollary:invariance}
Let $a\geq 1$, $\delta>0$, and $U\subset \widehat{\C}$ be a $\delta$-hyperbolic domain such that $\infty\in U$, $\widehat \C\setminus U\subset \br B_e(0,a)$, and $\diam_e(\widehat \C\setminus U)\geq a^{-1}$. If $f\colon U\to \widehat{\C}$ is a conformal embedding with $f(\infty)=\infty$ that satisfies the normalization \eqref{lemma:normal:normalization}, then $f(U)$ is $\delta'(a,\delta)$-hyperbolic.
\end{corollary}

\begin{proof}
By Lemma \ref{lemma:normal}, there exists a constant $b\geq1$ depending on $a$ such that $\widehat\C\setminus f(U)\subset \br B_e(0,b)$ and $\diam_e(\widehat{\C}\setminus f(U))\geq b^{-1}$. We have 
$$\diam_{\sigma}(\widehat \C\setminus U) \geq \diam_{\chi}(\widehat\C\setminus U)\geq \frac{2}{1+a^2} \diam_e(\widehat \C\setminus U)\geq \frac{2a^{-1}}{1+a^2}$$
and the same estimate is true for $\widehat \C\setminus f(U)$ in place of $\widehat \C \setminus U$ and $b$ in place of $a$. Theorem \ref{theorem:conformal_bilip} provides the desired conclusion.
\end{proof}

\subsection{Spherical inner uniform domains}
Recall the definition of a spherical inner uniform domain from the Introduction. 

\begin{theorem}\label{theorem:spherical_euclidean_inner}
Let $a>0$, $A\geq 1$, and $\widehat{\Omega}\subset \widehat{\C}$ be a spherical inner $A$-uniform domain with $\widehat{\C}\setminus \widehat \Omega \subset \br B_e(0,a)$. Then $\Omega=\widehat{\Omega}\cap \C$ is a Euclidean inner $C(a,A)$-uniform domain.
\end{theorem}

We start with a preliminary statement.

\begin{lemma}\label{lemmadistance}
Let $\widehat \Omega\subset \widehat{\C}$ be a domain such that  and $\widehat{\C}\setminus \widehat \Omega\subset \br B_e(0,a)$ for some $a>0$. Also, let $\Omega=\widehat \Omega \cap \C$. Then for all $z\in \Omega\cap \br B_e (0,a)$ we have
\[\dist_{\sigma} (z,\widehat \C\setminus \Omega)\le \dist_{\sigma}(z,\widehat \C\setminus \widehat \Omega)\le C(a)\dist_{\sigma} (z,\widehat \C\setminus \Omega) \leq 2C(a) \dist_{e}(z,\C\setminus \Omega).\] 
\end{lemma}

\begin{proof}
The first inequality is trivial since $\Omega\subset \widehat{\Omega}$. For the second, let $z\in \Omega\cap \br B_e (0,a)$. By comparing the spherical and chordal metrics, we deduce 
\begin{align*}
\dist_{\sigma} (z,\widehat \C\setminus \widehat\Omega) \le\frac{ \pi}{\sqrt{1+|z|^2}} \min\{{|z-\zeta|}:\zeta\in \widehat\C\setminus \widehat\Omega\}\le {\pi a }\sigma(z,\infty),
\end{align*}
since $|z-\zeta|\le 2a$. If $\dist_{\sigma}(z,\widehat\C\setminus \widehat\Omega)>\dist_\sigma(z,\widehat{\C}\setminus \Omega)$, then $\dist_\sigma(z,\widehat{\C}\setminus \Omega)=\sigma (z,\infty)$, so the above inequality proves the second inequality. The third inequality follows from the fact that $\dist_{\sigma }(z,\widehat{\C}\setminus \Omega)\leq \sigma(z,\zeta)\leq 2|z-\zeta|$ for every $\zeta\in \C\setminus \Omega$.
\end{proof}

In the proof below we use notions from Section \ref{section:inner_euclidean}, such as cigars, adapted to the spherical metric. The notation is self-explanatory.

\begin{proof}[Proof of Theorem \ref{theorem:spherical_euclidean_inner}]
We start with some basic facts. For each curve $\gamma$ in $\C$,
\begin{align}\label{theorem:se:lengths}
\textrm{if $|\gamma|\subset  \br B_e(0,3a)$, then $\ell_e(\gamma) \leq K(a) \ell_{\sigma}(\gamma)$},
\end{align}
where $K(a)= 2^{-1}(1+9a^2)$. Moreover, if $\gamma$ is an inner $A$-uniform curve in $\widehat \Omega$, then the condition $\cig_{\ell_{\sigma}}(\gamma, A)\subset \widehat{\Omega}$ implies that 
\begin{align}\label{theorem:se:2pi}
\ell_{\sigma}(\gamma)\leq 2\pi A.
\end{align}
Let $x,y\in \Omega$. We consider three cases.

\medskip
\noindent
\textbf{Case 1:} $x,y\in \Omega\setminus \br B_e(0,a)$. It is elementary to show that $\C\setminus \br B_e(0,a)$ is a Euclidean $C$-uniform domain for a constant $C>0$ independent of $a>0$. Thus, any two points  $x,y\in \Omega\setminus \br B_e(0,a)$ can be connected by a curve $\gamma$ in $\Omega$ with $\ell_e(\gamma)\leq C |x-y|\leq C\lambda^e_{\Omega}(x,y)$ and $\cig_{\ell_e}(\gamma, C)\subset \Omega$.

\medskip
\noindent
{\bf Case 2:}  $x,y\in \Omega\cap {B_e}(0,3a)$. We will show that there exists an inner $C(a,A)$-uniform curve $\gamma'$ in $\Omega$ connecting $x$ and $y$ such that 
\begin{align}\label{theorem:se:case2}
\ell_e(\gamma') \leq 4K(a) \pi A.
\end{align} 
By assumption there is a spherical inner $A$-uniform curve $\gamma \colon [0,1]\to \widehat \Omega$ connecting $x$ and $y$.   

\medskip
\noindent
{\bf Case 2(a):} $|\gamma| \subset  \Omega\cap \overline{B}_e(0,3a) $. By \eqref{theorem:se:lengths}, we have
\begin{align*}
\ell_e(\gamma)&\le K(a)\ell_{\sigma}(\gamma)\le  K(a)A\lambda_{\widehat \Omega}^{\sigma} (x,y)\le K(a)A\lambda_{\Omega}^{\sigma} (x,y)\le 2K(a)A\lambda_{\Omega}^e(x,y).
\end{align*}
In combination with \eqref{theorem:se:2pi}, this implies \eqref{theorem:se:case2}. For $t\in [0,1]$, by \eqref{theorem:se:lengths} and Lemma \ref{lemmadistance}, we have
\begin{align*}
\min\{\ell_e (\gamma|_{[0,t]}),\ell_e (\gamma|_{[t,1]})\}&\le K(a)\min\{\ell_{\sigma} (\gamma|_{[0,t]}),\ell_{\sigma} (\gamma|_{[t,1]})\} \\
&\le K(a)A\dist_{\sigma}(\gamma(t), \widehat \C \setminus \widehat \Omega)\\
&\le K(a)C(a)A\dist_e(\gamma(t),\C\setminus \Omega).
\end{align*}
So, $\gamma$ is a Euclidean inner $2K(a)C(a)A$-uniform curve connecting $x$ and $y$ in $\Omega$. 

\medskip
\noindent
{\bf Case 2(b):} $|\gamma| \not\subset \Omega\cap \overline{B}_e(0,3a) $. Let $z_1$ and $z_2$ be the first and last point, respectively, of the intersection of $\gamma$ with $\partial B_e(0,3a)$, assuming that $\gamma$ starts at $x$. We set $\gamma_1$ to be the subpath of $\gamma$ joining $x$ to $z_1$, $\gamma_2$ the subpath of $\gamma$ from $z_1$ to $z_2$ and $\gamma_3$ the subpath from $z_2$ to $y$. We define a curve $\gamma'$ to be the concatenation of $\gamma_1$, $\gamma_3$, and a path $\gamma_2'$ that is the smallest arc on $\partial B_e(0,3a)$ connecting $z_1$ and $z_2$. Then $|\gamma'|\subset \Omega\cap \overline{B}_e(0,3a)$. By \eqref{theorem:se:lengths}, we have
\begin{align*}
\ell_e (\gamma_2')\le \frac{\pi}{2}|z_1-z_2|\le \frac{\pi}{2}K(a)\sigma(z_1,z_2)\le 2K(a) \ell_{\sigma}(\gamma_2)
\end{align*}
and hence, again by \eqref{theorem:se:lengths},
\begin{align}\label{firstcond}
\ell_e (\gamma')&=\ell_e (\gamma_1)+\ell_e (\gamma_2')+\ell_e (\gamma_3)\le {2}K(a) (\ell_{\sigma} (\gamma_1)+\ell_{\sigma} (\gamma_2)+\ell_{\sigma} (\gamma_3)) \nonumber\\
&=2K(a) \ell_{\sigma} (\gamma)\le 2 K(a) A \lambda_{\Omega}^{\sigma} (x,y)\le 4 K(a) A\lambda_{\Omega}^e (x,y).
\end{align}
Note that by \eqref{firstcond} and \eqref{theorem:se:2pi} we obtain \eqref{theorem:se:case2}. Now, we check that  $\gamma_1$ and $\gamma_2'$ satisfy the assumptions of Lemma \ref{lemma:concatenation_cigar}. First, we have
\[\min\{\ell_e(\gamma_1),\ell_e (\gamma_2')\}\le \ell_e(\gamma_2')\le 6\pi a \le 3\pi \dist_e (z_1,\C\setminus \Omega).\]
Consider parametrizations $\gamma_2',\gamma_1\colon[0,1]\to \Omega$ such that $\gamma_1(0)=x$, $\gamma_1(1)=z_1$ and $\gamma_2'(0)=z_1$, $\gamma_2'(1)=z_2$. Then by \eqref{theorem:se:lengths} and Lemma \ref{lemmadistance}, for $t\in[0,1]$, we have
\begin{align}\label{gamma1}
\min \{\ell_e (\gamma_1|_{[0,t]}),\ell_e (\gamma_1|_{[t,1]})\}&\le K(a) \min \{\ell_{\sigma} (\gamma_1|_{[0,t]}),\ell_{\sigma} (\gamma_1|_{[t,1]})\} \nonumber \\
&\le K(a)A\dist_{\sigma} (\gamma_1(t),\widehat \C\setminus \widehat \Omega) \nonumber\\
&\le K(a)C(a)A\dist_e(\gamma_1(t),\C\setminus \Omega)
\end{align}
and trivially we have 
\begin{align*}
\min \{\ell_e (\gamma_2'|_{[0,t]}),\ell_e (\gamma_2'|_{[t,1]})\}\le \ell_e(\gamma_2') \leq 6\pi a\leq  3\pi \dist_e (\gamma_2'(t),\C\setminus \Omega).
\end{align*}
Hence, Lemma \ref{lemma:concatenation_cigar} implies that  $\cig_{\ell_e} (\gamma_{12},C_1(a,A))\subset \Omega$, where $\gamma_{12}$ is the concatenation of $\gamma_1$ and $\gamma_2'$ at $z_1$. Working as in (\ref{gamma1}) we obtain $\cig_{\ell_e} (\gamma_3,K(a)C(a) A)\subset \Omega$. Also, by \eqref{theorem:se:lengths} and \eqref{theorem:se:2pi}, we obtain
\begin{align*}
\min\{\ell_e(\gamma_{12}),\ell_e (\gamma_3)\}&\le \ell_e(\gamma_3)\leq K(a)\ell_{\sigma}(\gamma)\le K(a)2\pi A \leq  \frac{K(a)\pi A}{a}\dist_e (z_2,\C\setminus \Omega).
\end{align*}
Lemma \ref{lemma:concatenation_cigar} implies that  $\cig_{\ell_e} (\gamma',C_2(a,A))\subset \Omega$. By this and (\ref{firstcond}) we conclude that  $\gamma'$ is a Euclidean inner $C_3(a,A)$-uniform curve connecting $x$ to $y$ in $\Omega$.

\medskip
\noindent
{\bf Case 3:} $x\in \Omega\cap {B}_e(0,2a)$ and $y\in \Omega\setminus {B}_e(0,3a)$.  Let $z\in \partial B_e(0,2a)$ be an arbitrary point. By Case 2 there exists a Euclidean inner $C(a,A)$-uniform curve $\gamma_x$ in $\Omega$ connecting $x$ to $z$ and $\ell_e(\gamma_x)\leq 4K(a)\pi A$. Since $\C\setminus \br B_e(0,a)$ is a Euclidean $C$-uniform domain, there exists a Euclidean inner $C$-uniform curve $\gamma_y$ in $\Omega$ connecting $y$ to $z$. Note that
\[\min\{\ell_e (\gamma_x), \ell_e (\gamma_y)\}\le \ell_e(\gamma_x)\le 4K(a)\pi A\leq \frac{4K(a)\pi A}{a}  \dist (z,\C \setminus \Omega).\]
Thus, we may apply Lemma \ref{lemma:concatenation_cigar} to obtain that the concatenation $\gamma$ of $\gamma_x$ and $\gamma_y$ is a Euclidean inner $C_4(a,A)$-uniform curve joining $x$ and $y$ in $\Omega$.
\end{proof}

\subsection{Proof of Theorem \ref{theorem:main}}

A domain $\Omega\subset \widehat{\C}$ is a \textit{slit domain} if $\infty \in \Omega$ and all components of $\widehat{\C}\setminus \Omega$ are either horizontal line segments or vertical line segments. Some segments could be degenerate, so $\widehat{\C}\setminus \Omega$ can also have point components. 
We use the following statement, which is a consequence of \cite{BonkHeinonenKoskela:gromov_hyperbolic}*{Proposition 7.12}.
\begin{prop}\label{prop:uniform_circle_hyperbolic}
Let $\Omega\subset \widehat{\C}$ be a $\delta$-hyperbolic domain for some $\delta>0$. 
\begin{enumerate}[label=\normalfont(\roman*)]
	\item If $\Omega$ is a circle domain, then $\Omega$ is a spherical $A(\delta)$-uniform domain. 
	\item If $\Omega$ is a slit domain, then $\Omega$ is a spherical inner $A(\delta)$-uniform domain.
\end{enumerate}
\end{prop}

\begin{proof}[Proof of Theorem \ref{theorem:main}]
Let $\delta>0$ and $\Omega\subset \widehat{\C}$ be a $\delta$-hyperbolic domain. If $\partial \Omega$ is a single point, then $\Omega$ is already a circle domain that is uniform so we have nothing to prove. Suppose that $\partial \Omega$ contains at least two points. Then the largest spherical ball contained in $\Omega$ has radius less than $\pi$. Let $a\in (0,\pi)$ and suppose that the largest spherical ball contained in $\Omega$ has radius in the interval $[a,\pi-a]$. We will show that $\Omega$ is a conformally equivalent to an $A$-uniform circle domain, where $A\geq 1$ depends only on $\delta$ and $a$.

By applying a spherical isometry, we assume that the largest spherical ball contained in $\Omega$ is centered at $\infty$. Hence, each point $z\in \widehat \C\setminus \Omega$ satisfies $\sigma(0,z)\leq \pi-a$, so  $\widehat \C\setminus \Omega$ is contained in $\br B_e(0,\tan(\pi/2-a/2))$. On the other hand, $\widehat \C\setminus \Omega$ is not contained  in a spherical ball of radius $a$. In particular, 
$$ \diam_{e}(\widehat \C\setminus \Omega)\geq \frac{1}{2}\diam_{\sigma}( \widehat \C\setminus \Omega)\geq \frac{a}{2}.$$ 

By \cite{Schiff:normal_families}*{Theorem 5.1.5}, there exists a conformal map $h$ from $\Omega$ onto a slit domain $\widehat U$ such that $h(\infty)=\infty$ and $h$ satisfies the normalization in \eqref{lemma:normal:normalization}. By Lemma \ref{lemma:normal}, there exists $b\geq 1$ depending only on $a$ such that $\widehat \C\setminus \widehat U\subset \br B_e(0,b)$ and $\diam_e(\widehat \C\setminus \widehat U)\geq b^{-1}$. By Corollary \ref{corollary:invariance}, $\widehat U$ is $\delta_1$-hyperbolic for some $\delta_1>0$ that depends only on $\delta$ and $a$. By Proposition \ref{prop:uniform_circle_hyperbolic}, there exists $A_1\geq 1$ that depends only on $\delta_1$ such that $\widehat U$ is a spherical inner $A_1$-uniform domain. By Theorem \ref{theorem:spherical_euclidean_inner}, $U=\widehat{U}\cap \C$ is a Euclidean inner $A_2$-uniform domain for some $A_2\geq 1$ that depends only on $A_1$ and on $b$.

By Theorem \ref{theorem:uniformization_euclidean}, there exists a conformal map $f$ from $U$ onto a circle domain $D\subset \C$ that contains a neighborhood of $\infty$ and $f$ extends to $\infty$ so that $f(\infty)=\infty$. By postcomposing with a conformal automorphism of $\C$, we may assume that $f$ satisfies the normalization in \eqref{lemma:normal:normalization}. By Corollary \ref{corollary:invariance}, we conclude that $\widehat D=f(\widehat U)$ is $\delta_2$-hyperbolic for some $\delta_2>0$ depending only on $\delta_1$ and $b$. We employ again Proposition \ref{prop:uniform_circle_hyperbolic}, which implies that $\widehat D$ is a spherical $A_3$-uniform circle domain for some $A_3\geq 1$ that depends only on $\delta_2$. The composition $g=f\circ h\colon \Omega\to \widehat D$ gives the desired conformal map and completes the proof of the existence in Theorem \ref{theorem:main}.  The uniqueness follows from the uniqueness part of Theorem \ref{theorem:uniformization_euclidean}.
\end{proof}

\bibliography{biblio} 

\end{document}